\documentclass[12pt]{amsart}
\usepackage{amssymb}

\input xy
\xyoption{all}  

\theoremstyle{plain}
\newtheorem{prop}{Proposition}[section]
\newtheorem{thm}[prop]{Theorem}
\newtheorem{coro}[prop]{Corollary}
\newtheorem{lemm}[prop]{Lemma}

\newtheorem*{thm-int}{Theorem}
\theoremstyle{remark}
\newtheorem{exam}[prop]{Example}
\newtheorem{rema}[prop]{Remark}

\theoremstyle{definition}
\newtheorem{defn}[prop]{Definition}

\numberwithin{equation}{section}

\def\Gal{{\it G}}
\def\G{{\mathcal G}}   % abel. pro-l-Galois group 
\def\D{{\mathcal D}}   % decomposition group
\def\I{{\mathcal I}}   % inertia groups

\def\Val{{\mathcal V}}
\def\Spec{{\mathrm{Spec}}}

\def\DVal{{\mathcal D}{\mathcal V}}

\def\KK{\boldsymbol{K}}

\def\no{\noindent}
\def\rk{{\rm rk}}
\newcommand{\trdeg}{{\rm tr}\, {\rm deg}}

\def\rR{{\mathrm R}}

\newcommand{\Hom}{{\rm Hom}}
\newcommand{\Ker}{{\rm Ker}}

\def\lra{\longrightarrow}
\def\ra{\rightarrow}

\def\F{{\mathbb F}}

\def\P{{\mathbb P}}
\def\Q{{\mathbb Q}}

\def\Z{{\mathbb Z}}

\def\N{{\mathbb N}}

\def\cZ{{\mathcal Z}}

\def\bP{{\mathbb P}}

\def\GL{{\rm GL}}

\def\Ind{{\rm Ind}}

\def\rK{{\mathrm K}}

\def\rH{{\rm H}}

\def\rI{{\rm I}}

\def\Syl{{\rm Syl}}

\begin{document}
\title[Universal spaces]{Universal spaces for unramified Galois cohomology}
\author{Fedor Bogomolov}
\address{Courant Institute of Mathematical Sciences, N.Y.U. \\
 251 Mercer str. \\
 New York, NY 10012, U.S.A.\\
 ~Laboratory of Algebraic
Geometry\\
National Research University Higher School of Economics\\
7 Vavilova str.\\
Moscow, 117312, Russia}
\email{bogomolo@cims.nyu.edu}

\author{Yuri Tschinkel}
\address{Courant Institute of Mathematical Sciences, N.Y.U. \\
 251 Mercer str. \\
 New York, NY 10012, U.S.A.}
\email{tschinkel@cims.nyu.edu}

\address{Simons Foundation\\
160 Fifth Avenue\\
New York, NY 10010\\
USA}


\begin{abstract}
We construct and study 
universal spaces for birational invariants of algebraic varieties over algebraic closures of finite fields.  
\end{abstract}

\maketitle

\setcounter{section}{0}       
\section*{Introduction}
\label{sect:introduction}

Let $\ell$ be a prime. Recall that in topology, there exist unique (up to homotopy) 
topological spaces $\rK(\Z/\ell^n,m)$ such that 
\begin{itemize}
\item $\rK(\Z/\ell^n,m)$ is homotopically trivial up to dimension $m-1$, 
in particular, 
$$
\rH^i(\rK(\Z/\ell^n,m), \Z/\ell^n)=0, \quad  \text{ for } \quad i < m;
$$
\item $\rH^m(\rK(\Z/\ell^n,m), \Z/\ell^n)$ is cyclic, with a distinguished
generator $\kappa_m$;
\item for every topological space $X$ and every 
$\alpha \in \rH^m(X,\Z/\ell^n)$ there is a unique, 
up to homotopy, continuous map 
$$
\mu_{X,\alpha} \,:\,  X\ra \rK(\Z/\ell^n,m)
$$
such that 
$$
\mu_{X,\alpha}^*(\kappa_m)=\alpha.
$$ 
\end{itemize}
This reduces many questions about singular cohomology 
to the study of these universal spaces. 
Analogous theories exist for other contravariant functors, for example, 
topological $\rK$-theory, or the theory of cobordisms. 
The study of moduli spaces in algebraic geometry 
can be viewed, broadly speaking, 
as an incarnation of the same idea of universal spaces.

Here we propose a similar theory for unramified cohomology, 
developed in connection with the study of birational
properties of algebraic varieties \cite{bog-stable}, \cite{colliot-ojan}. 
The Bloch--Kato conjecture proved 
by Voevodski, Rost, and Weibel,
combined with techniques and results
from birational anabelian geometry  in \cite{bt0},
implies that an unramified class in the
cohomology of a function field $K=k(X)$ of an algebraic variety $X$ 
over an algebraic closure of a finite field 
$k=\bar{\mathbb F}_p$, with finite constant coefficients,  
is induced from the cohomology of 
a finite abelian group $G^a$. 
Our main result is:

\begin{thm-int}
\label{thm:main-intro}
Let $\ell$ and $p$ be distinct primes,
$K=k(X)$ the function field of an algebraic variety 
$X$ of dimension $\ge 2$ over $k=\bar{\mathbb F}_p$, $G_K$ 
its absolute Galois group, 
and $\alpha_K\in \rH^*_{nr}(G_K,\Z/\ell^n)$ an unramified class.
Then there exists a finite set $J$ 
of finite-dimensional $k$-vector spaces $V_j, j\in J$, such that 
$\alpha_K$ is induced, via a rational map, from an unramified class 
in the cohomology of an explicit open subset of the quotient of 
$$
\mathbb P:=\prod_{j\in J} \mathbb P(V_j)
$$ 
by a finite abelian $\ell$-group $G^a$, 
acting projectively on each factor.
\end{thm-int}

Thus, the spaces $\mathbb P/G^a$ serve as universal 
spaces for {\em all} finite birational invariants 
of algebraic varieties over $k=\bar{\mathbb F}_p$. 

Actions of finite abelian groups $G^a$ on products of projective spaces 
are described by central extensions of $G^a$, i.e., 
by subspaces in $\wedge^2(G^a)$. This allows to 
present unramified classes of $X$ in terms of 
configurations of subspaces of skew-symmetric matrices. 
For example, if the unramified Brauer group of $X$ is trivial, then 
all finite birational invariants of $X$  
are encoded already in the combinatorics of 
configurations of liftable
subgroups in finite abelian quotients of 
the absolute Galois group $G_K$  (see 
Section~\ref{sect:general} for the definition).

Similar ideas were first put forward in \cite{bog-stable} and \cite{bog-goe},
however, the recent proof of the Bloch--Kato conjecture 
allows us to formulate a more precise and constructive theory. 
This approach to birational invariants leads to many new questions:

\begin{itemize}
\item
Is there a smaller class of configurations with this universal property?
\item 
How does this structure interact with Sylow subgroups of $G_K$?
\item 
Is there an extention to cohomology with $\Z_{\ell}$-coefficients?
An equally simple
description of models for $\ell$-adic 
invariants would provide insights into higher-dimensional
Langlands correspondence.                          
\item 
What are the analogs of universal spaces for varieties over $k=\bar{\Q}$?
Counterexamples to our main result arise from bad reduction places, 
already for abelian varieties \cite{bog-stable}. 
\end{itemize}

Here is the roadmap of the paper: In Section~\ref{sect:general} we recall basic facts about 
stable and unramified cohomology. In Section~\ref{sect:valuation} we provide some background on 
valuation theory.  
In Section~\ref{sect:galois} we investigate Galois cohomology groups of function fields of higher-dimensional 
algebraic varieties over $k=\bar{\mathbb F}_p$ and their images in cohomology of finite groups. 
In Section~\ref{sect:unram-cohomology} we introduce and study unramified cohomology of algebraic varieties. 
Section~\ref{sect:proof} contains the proof of our main theorem, 
modulo geometric considerations presented in Sections~\ref{sect:smooth} and \ref{sect:singular}.

\

\no
{\bf Acknowledgments.}
We are grateful to A. Pirutka for her interest and insightful comments.  
The first author was supported by NSF grant DMS-1001662 and by AG Laboratory GU-HSE grant
RF government ag. 11 11.G34.31.0023.
The second author was partially supported by NSF grants DMS-0739380, 
0901777, and 1160859.

\section{Stable cohomology}
\label{sect:general}

Let $G$ be a pro-finite group.
We will write 
$$
G^a=G/[G,G] \quad \text{ and } \quad G^c = G/[[G,G],G]
$$
for the abelianization, respectively, the second lower central series quotient of $G$. 
We have a canonical central extension
\begin{equation}
\label{eqn:centrall}
1\ra Z\ra G^c\stackrel{\pi_a}{\longrightarrow} G^a\ra 1.
\end{equation}

Let $M$ be a topological $G$-module and
$\rH^i(G,M)$ its (continuous) $i$-cohomology group.
These groups are contravariant with respect to $G$ and covariant
with respect to $M$. 
In this paper, $G$ is either a finite group
or a Galois group (see \cite{milgram} for background on group cohomology and 
\cite{serre} for background on Galois cohomology) and $M$ either $\Z/\ell^n$ or $\Q/\Z$,
with trivial $G$-action. We will sometimes omit the coefficient module $M$ from 
the notation. 

Our goal is to investigate incarnations of Galois cohomology of 
function fields in cohomology of finite groups.
For example, let $K=k(X)$ be a function field of an 
algebraic variety $X$ over an algebraically closed field $k$; 
varieties birational to $X$ are called {\em models} of $K$.
Let $G_K$ be the absolute Galois group of $K$ and $\hat{\pi}_1(X)$ the \'etale 
fundamental group of $X$, with respect to some basepoint. The choice of a  base point 
will not affect our considerations and we omit it from our notation. 
We have natural homomorphisms
$$
\rH^*(\hat{\pi}_1(X))\stackrel{\kappa^*_X}{\lra} \rH^*_{et}(X) \stackrel{\tilde{\eta}^*_X}{\lra} \rH^*(G_K),
$$
where the right arrow arises from the embedding of the generic point $X_{\eta}\ra X$. 
We will write
$$
\eta_X:G_K \ra \hat{\pi}_1(X)
$$
and 
$$
\eta^*_X=\tilde{\eta}^*_X\circ \kappa_X^*:  \rH^*(\hat{\pi}_1(X))\ra \rH^*(G_K)
$$
for the corresponding map in cohomology. 

We say that a class $\alpha_K\in \rH^*(G_K)$ is defined (or represented) on a model $X$ of $K$ if there exists a class 
$\alpha_X\in \rH^*_{et}(X)$ such that
$$
\alpha_K=\tilde{\eta}_X^*(\alpha_X).
$$
Let $G$ be a finite group. A continuous homomorphism
$$
\chi: \hat{\pi}_1(X) \ra G
$$
gives rise to homomorphisms in cohomology
$$
\rH^*(G)\stackrel{\chi^*}{\lra} \rH^*(\hat{\pi}_1(X)) \stackrel{\eta_X^*}{\lra} \rH^*(G_K).
$$
Conversely, every $\alpha_K\in \rH^*(G_K)$ arises in this way:  
there exist
\begin{itemize}
\item a model $X$ of $K$,
\item a continuous homomorphism $\chi$ as above,
\item and a class 
$\alpha_G\in \rH^*(G)$
\end{itemize}
such that
$$
\alpha_K=\eta_X^*(\chi^*(\alpha_G)).
$$
This follows from the description of \'etale cohomology of points, see \cite{milne}.
In such situations we say that $\alpha_K$ is defined on $X$ and is induced 
from $\chi$.

A version of this construction arises as follows:
assume that the characteristic of $k$ does not divide the order of $G$. 
Let $V$ be a faithful representation of $G$ over $k$, and  
$X$ an algebraic variety over $k$ with function field 
\[
K=k(X)\simeq k(V)^G,
\] 
the field of invariants; we will write $X=V/G$ and call it a {\em quotient}. 
Even more generally, let 
$Y$ be an algebraic variety over $k$ with a generically free action of $G$,
$X=Y/G$ a quotient.
This situation gives rise to a natural surjective continuous homomorphism
\[
G_K\ra G
\]
and induced homomorphisms on cohomology 
$$
s^i_K \,:\, \rH^i(G)\ra \rH^i(G_K).
$$
The following lemma shows that we have many choices in realizing a class 
$\alpha_K\in \rH^i(G_K)$:

\begin{lemm}
\label{lemm:vg}
\cite{bog-stable}
Assume that $\alpha_K\in \rH^i(G_K)$ is represented by a class $\alpha_X\in \rH^i_{et}(X)$ on 
some affine irreducible model $X$ of $K$ and is induced
from a surjective continuous homomorphism $\chi:\hat{\pi}_1(X)\ra G$ and a class $\alpha_G\in \rH^i(G)$. 
Let $V$ be a faithful representation of $G$ over $k$ and 
$V^\circ\subset V$
the locus where the action is free.
Then, for every $x\in X$ and $v\in V^\circ$ there exists a map 
$$
f=f_{x} : X\to V/G
$$
such that
\begin{itemize}
\item $f(x)= v$ and 
\item the restriction of $\alpha_X$ to $X^\circ= f^{-1}(V^\circ/G)\subset X$ 
is equal to $f^*(\alpha_G)$.
\end{itemize}
\end{lemm}

\begin{proof}
We follow the proof in \cite{bog-stable}.
The homomorphism $\chi:\hat{\pi}_1(X)\to G$ defines a finite
\'etale covering $\pi:\tilde X\to X$, by an affine variety
$\tilde{X}$. The ring $k[\tilde X]$ is 
a $k[G]$ -module. Every finite-dimensional $k[G]$-submodule 
$W^*\subset k[\tilde X]$ defines
a $G$-equivariant map $\tilde X\to W:=\Spec(W^*)$.

Let $e\in k[G]$ be the unit element of $G$.
For any $G$-orbit $G\cdot y\in V$ there is a $G$-linear homomorphism 
$$
l_y : k[G]\to V,
$$
which maps the orbit $G\cdot e$ to $G\cdot y$. Let $\tilde{x}\in \pi^{-1}(x)$. 
Choose $h\in k[\tilde{X}]$ such that 
$$
h(\tilde{x})=1, \quad h(g \cdot \tilde{x})=0, \quad g\neq e.
$$
Then $h$ generates a $k[G]$-submodule $W\subset k[\tilde X]$
and defines a regular $G$-map $h: \tilde X\to W=k[G]$,
with $h(\tilde{x}) = e\in k[G]$.
The map $f := l_y \circ h $ is a regular $G$-map satisfying the first property.

Let $X^\circ= f^{-1}(V^\circ/G)\subset X$. It is a nonempty affine subvariety.
We have a compatible diagram of $G$-maps

\centerline{
\xymatrix{
V^\circ \ar[d]_{\pi_G} & \ar[l]\tilde{X}^\circ \ar[d]^{\pi_0} & \subset & \tilde X\ar[d]^{\pi}\\
V^\circ/G   & \ar[l]^{\,\,\, f}X^\circ                & \subset & X 
}
}

\noindent
and the maps $\pi_G$ and $\pi$ induce the same cover $\pi_0$. This
implies the second claim.
\end{proof}

We can  achieve even more flexibility for $G$-maps, 
under a {\em projectivity} conditions on $V$: 
we say that a $G$-module $V$ is projective if 
for every finite-dimensional representation $W$ of $G$
with a $G$-surjection 
$$
\mu : W\to V
$$
there exists a $G$-section
$$
\theta :V\to W\quad \text{ with } \quad \mu\circ \theta ={\rm id}.
$$
This condition holds, for example, for regular representations over
arbitrary fields or when the order of $G$ is coprime to the 
characteristic of the ground field $k$.

Now let $\{S_j\}_{j\in J}$ be a finite set of $G$-orbits in
$Y$ with stabilizers $H_j$ so that $S_j\simeq G/H_j$.
Consider a faithful representation $V$ of $G$ and a subset
$\{T_j\}_{j\in J}$ of $G$-orbits in $V$ with stabilizers $Q_j$, with $H_j\subset Q_j$, $T_j=G/Q_j$. 
Consider regular $G$-maps $f_j: S_j\to T_j$, for $j\in J$.

\begin{lemm}
Assume that $V$ is a projective $G$-module. Then 
there is a regular
$G$-map $f: Y\to V$ such that $f=f_j$, for all $j\in J$.
\end{lemm}

\

We return to our setup: $X=Y/G$, $K=k(X)$, and $\chi: G_K\ra G$, inducing 
$$
s^i_K: \rH^*(G)\ra \rH^*(G_K).
$$ 
The groups
$$
\rH^i_{s,K}(G) := \rH^i(G) / \Ker(s^i_K)
$$
are called {\em stable cohomology groups} with respect to $K=k(X)$.  
Let 
$$
\Ker(s^i):=\bigcap_{K}\Ker(s^i_K),
$$
over all function fields $K=k(X)$ as above. 
In fact, 
$$
\Ker(s^i)= \Ker(s^i_{k(V/G)}),
$$
for some faithful representation $V$ of $G$ over $k$, in particular, this is independent
of the choice of $V$  (see \cite[Proposition 4.3]{bpt}).
The groups
$$
\rH^i_{s}(G) := \rH^i(G) / \Ker(s^i)
$$
are called {\em stable cohomology groups} of $G$ (with coefficients in 
$M=\Z/\ell^n$ or $\Q/\Z$);
they depend on the ground field $k$.
These define contravariant functors in $G$.
For example, for a subgroup $H\subset G$ we have a restriction homomorphism
$$
\mathrm{res}_{G/H}: \rH^*_s(G)\ra \rH^*_s(H).
$$
Furthermore:

\begin{itemize}
\item 
While usual group cohomology $\rH^i(G)$ can be nontrivial for infinitely many $i$
(even for cyclic groups), {\em stable} cohomology groups $\rH^i_s(G)=0$ for $i > \dim(V)$, 
where $V$ is a faithful representation.  
\item We have
$$
\rH^i_s(G) \subseteq \rH^i_s(\Syl_{\ell}(G))^{N_G(\Syl_{\ell}(G))},
$$ 
where $N_G(H)$ is the normalizer of $H$ in $G$ and $\Syl_{\ell}(G)$ an $\ell$-Sylow subgroup of $G$.  
\end{itemize}

The determination of the stable cohomology ring
$$
\rH^*_s(G) := \oplus_i \rH^i_s(G)
$$
is a nontrivial problem, see, e.g., \cite{b-b} 
for a computation of stable cohomology of alternating groups.  
For finite abelian groups $G$, we have
\begin{equation}
\label{eqn:wehave}
\rH^*_s(G)=\wedge^*(\rH^1(G)).
\end{equation}
For central extensions of finite groups as in \eqref{eqn:centrall}, 
the kernel of 
$$
\pi_a^*:\rH^*_s(G^a)\ra \rH^*_s(G^c)
$$
is the ideal $\rI=\rI(G^c)$ generated by 
$$
\rR^2=\rR^2(G^c):=\Ker\left(\rH^2_s(G^a)\ra \rH^2_s(G^c)\right).
$$
(see, for example, \cite[Section 8]{bt-msri}).
We will identify 
$$
\rH^*_s(G^a) / \rI
$$
with its image in $\rH^*_s(G^c)$. 
An important role in the computation of 
this subring of $\rH^*_s(G^c)$ is played by the {\em fan}
$$
\Sigma=\Sigma(G^c)=\{\sigma\},
$$
the set of {\em noncyclic} liftable subgroups 
$\sigma$ of $G^a$, and the {\em complete fan} 
$$
\bar{\Sigma}=\bar{\Sigma}(G^c) =\{\sigma\},
$$
consisting of {\em all} liftable subgroups $\sigma\subset G^c$:
a subgroup $\sigma$ is {\em liftable} if and only if the full preimage $\tilde{\sigma}$ of $\sigma$ in 
$G^c$ is abelian. 
The fan $\Sigma$ defines a subgroup 
$\rR^2(\Sigma)\subseteq \rH^2_s(G^a)$
as the set of all elements which vanish upon restriction to every $\sigma \in \Sigma$. 
We have 
$$
\rR^2\subseteq \rR^2(\Sigma).
$$

\begin{lemm}
\label{lemm:what-for}
For every $\alpha \in \rI(G^c)\subseteq \rH^*_s(G^a)$ and every $\sigma \in \Sigma(G^c)$ 
the restriction of $\alpha $ to $\sigma$ is trivial. 
\end{lemm}

%{\bf Proof or reference.}

\begin{defn}
\label{defn:delta}
Let 
$$
1\ra Z\ra G^c\ra G^a\ra 1
$$
be a central extension. A $\Delta$-pair $(I,D)$ of $G^a$ is a set of subgroups 
$$
I\subseteq D\subseteq G^a
$$
such that 
\begin{itemize}
\item $I\in \bar{\Sigma}(G^c)$, 
\item $D$ is noncyclic,
\item for every $\delta\in D$, the subgroup
$\langle I , \delta\rangle \in \bar{\Sigma}(G^c)$.
\end{itemize}
\end{defn}

This definition depends on $G^c$. Assume we have a commutative diagram of central extensions

\

\centerline{
\xymatrix{ 
1\ar[r] & \tilde{Z}\ar[d] \ar[r] & \tilde{G}^c \ar@{>>}[d] \ar[r] & \tilde{G}^a \ar@{>>}[d]^{\gamma} \ar[r] & 1 \\
1\ar[r] & Z \ar[r]         & G^c     \ar[r]     & G^a                \ar[r] & 1.
}
}

\

\begin{defn}
\label{defn:surjects}
A $\Delta$-pair $(\tilde{I},\tilde{D})$ of $\tilde{G}^a$ 
{\em surjects} onto a $\Delta$-pair $(I,D)$ of $G^a$ if
%\begin{itemize}
%\item 
$\gamma(\tilde{I})=I$  and
%\item 
$\gamma(\tilde{D})= D$.
%\end{itemize}
\end{defn}

\begin{defn}
\label{defn:unram}
A class $\alpha\in \rH^i_s(G^a)$ is 
{\em unramified} with respect to a $\Delta$-pair
if its restriction to $D$ is induced from $D/I$, i.e., 
there exists a $\beta\in \rH^i_s(D/I)$ such that $\phi(\alpha)=\psi(\beta)$, for the natural homomorphisms in the diagram:

\

\centerline{
\xymatrix{
\rH^*_s(G^a) \ar[r]^{\phi}  & \rH^*_s(D)  & \ar[l]_{\psi}\rH^*_s(D/I) 
}
}   
\end{defn}

By Lemma~\ref{lemm:what-for} we have a similar notion for 
$$
\rH^*_s(G^a)/\rI(G^c)\subseteq \rH^*_s(G^c).
$$

\begin{lemm}
\label{lemm:unram-delta}
Consider a homomorphism
$$
\gamma: \tilde{G}^a\ra G^a
$$ 
and a class $\alpha\in \rH^i(G^a)$. Let $\tilde{\alpha}:=\gamma^*(\alpha)\in \rH^i(\tilde{G}^a)$ be the induced class. 
Let $(\tilde{I}, \tilde{D})$ be a $\Delta$-pair in $\tilde{G}^a$. 
Assume that one of the following holds:
\begin{itemize}
\item $\gamma(\tilde{I})=0$, 
\item $\gamma(\tilde{D})$ is cyclic, 
\item $\gamma$ induces a surjection of $\Delta$-pairs
$$
(\tilde{I},\tilde{D})\ra (I,D)
$$ 
and $\alpha\in \rH^i_s(G^a)$ is unramified with respect to $(I,D)$.
\end{itemize}
Then $\tilde{\alpha}\in \rH^i_s(\tilde{G}^a)$ is unramified 
with respect to $(\tilde{I},\tilde{D})$. 
\end{lemm}

\begin{proof}
The first two cases are evident. Consider the third condition. 
By assumption, $\gamma$ induces a homomorphism 
$\tilde{D}/\tilde{I}\ra D/I$. Passing to cohomology we get a 
commuting diagram

\

\centerline{
\xymatrix{
\rH^*_s(D/I)          \ar[d] \ar[r]  & \rH^*_s(D) \ar[d]    \\
\rH^*_s(\tilde{D}/\tilde{I}) \ar[r] & \rH^*_s(\tilde{D}),
}
}

\

\noindent
and thus the claim.
\end{proof}

\section{Central extensions and isoclinism}
\label{sect:cent-ext}

Let $G^a$ and $Z$ be finite abelian  $\ell$-groups.
Central extensions of $G^a$ by $Z$ are parametrized by 
$\rH^2(G^a,Z)$; 
for $\alpha \in \rH^2(G^a,Z)$ we let $G^c_{\alpha}$ be the 
corresponding central extension:
\begin{equation}
\label{eqn:iso-alpha}
1\ra Z\ra G^c_{\alpha}\stackrel{\pi_a}{\lra} G^a\ra 1
\end{equation}
Fix an embedding $Z\hookrightarrow (\Q/\Z)^r$, consider 
the exact sequence
$$
1\ra Z\ra (\Q/\Z)^r \ra (\Q/\Z)^r \ra 1,
$$
and the induced long exact sequence in cohomology
$$
\rH^1(G^a,(\Q/\Z)^r) \stackrel{\delta}{\lra} \rH^2(G^a,Z) \ra \rH^2(G^a,(\Q/\Z)^r).
$$
We say that $\alpha, \tilde{\alpha}\in \rH^2(G^a,Z)$ and the corresponding 
extensions are {\em isoclinic} if 
$$
\alpha - \tilde{\alpha}\in \delta(\rH^1(G^a,(\Q/\Z)^r)).
$$
This notion does not depend on the chosen embedding $Z\hookrightarrow (\Q/\Z)^r$
and coincides with the standard definition of 
isoclinic in the theory of $\ell$-groups. 

\begin{lemm}
\label{lemm:iso-delta-pair}
If $\alpha, \tilde{\alpha}\in \rH^2(G^a, Z)$ are isoclinic then the corresponding extensions of $G^a$ 
define the same set of $\Delta$-pairs in $G^a$. 
\end{lemm}

\begin{proof}
A pair of subgroups $(I,D)$ is a $\Delta$-pair in $G^a$, with respect to a central extension $G^c$, 
if their preimages commute in $G^c$, i.e.,  
$$
[\pi_a^{-1}(I), \pi_a^{-1}(D)]=0 \quad \text{  in } \quad Z.
$$
Consider the homomorphism 
$$
\pi_a^*: \rH^2(G^a, \Q/\Z)\to \rH^2(G^c, \Q/\Z),
$$
and note that $\Ker(\pi_a^*)$ only depends on the isoclinism class of the extension. 
Furthermore, $\rH^2(G^a, \Q/\Z)$ is dual to $\wedge^2(G^a)$. 
Let $R\subset \wedge^2(G^a)$ be the subgroup 
which is dual to $\Ker(\pi_a^*)$. 
It remains to observe that $(I,D)$ is a $\Delta$-pair for $G^c$ 
if and only if $\pi_a^{-1}(I)\wedge \pi_a^{-1}(D)$ 
intersects $R$ trivially; thus the notion of a $\Delta$-pair is an invariant of the  
isoclinism class of the extension. 
\end{proof}

\begin{lemm}
\label{lemm:bock}
If $\alpha, \tilde{\alpha}\in \rH^2(G^a,Z)$ are isoclinic then
$$
\rR^2(G^c_{\alpha}) = \rR^2(G^c_{\tilde{\alpha}}),
$$
in particular
$$
\rH^*_s(G^a,\Z/\ell^n) / \rI(G^c_{\alpha}) = \rH^*_{s}(G^a,\Z/\ell^n) /\rI(G^c_{\tilde{\alpha}}).
$$
\end{lemm}

\begin{proof}
See Section 2 and Theorem 3.2 in \cite{b-b-iso}.
\end{proof}

\begin{lemm}
\label{lemm:det-bock}
If $\rR^2 = \rR^2(\Sigma)$ 
then $\Sigma$ determines $G^c$ up to isoclinism.
\end{lemm}

\begin{exam}
Consider $G^a={\mathbb F}_\ell^{2m}$, 
as a vector-space over $\mathbb F_{\ell^2}$ of dimension $m$. 
Let $\Sigma$ be the set of $\F_{\ell}$-linear subspaces of the form 
$\F_{\ell^2}$. Then 
$$
\rH^i_{s}(G^c,\Z/\ell^n) = 0, \quad \text{ for  } \quad i >2.
$$ 
\end{exam}

\begin{lemm}
\label{lemm:explicit}
If $\alpha, \tilde{\alpha}\in \rH^2(G^a,Z)$ are isoclinic
then there exist faithful representations $V,\tilde{V}$ of 
$G^c_{\alpha}$ and $G^c_{\tilde{\alpha}}$ over $k$
such that $V/G^c_{\alpha}$ and $\tilde{V}/G^c_{\tilde{\alpha}}$ are birational.
\end{lemm}

\begin{proof}
Explicit construction:
Let $\chi_1,\ldots, \chi_r$ be a basis of $\Hom(Z,k^\times)$ and put
$$
V:=\oplus_{j=1}^r V_j\ \quad \text{ and } \quad \tilde{V}=\oplus_{j=1}^{r} \tilde{V}_j,
$$
where
$$
V_j=\Ind_Z^{G^c_{\alpha}} (\chi_j) \quad \text{ and } \quad 
\tilde{V}_j=\Ind_Z^{G^c_{\tilde{\alpha}}} (\chi_j).
$$
Note that the {\em projectivizations}  $\P(V_j)$ and $\P(\tilde{V}_j)$ are canonically isomorphic
as $G^a$-representations.  
The group $(k^\times)^r$ acts on $V$ and $\tilde{V}$, 
and {\em both} $V/G^c_{\alpha}$ and $V/G^c_{\tilde{\alpha}}$ 
are birational to 
$$
\left(\prod_{j=1}^r  \P(V_j)\right) /G^a \times  
\left(\prod_{j=1}^r  k^\times / \chi_j(Z)\right). 
$$
\end{proof}

\begin{lemm}
\label{lemm:modify-g}
Consider a central extension of finite $\ell$-groups
\begin{equation}
\label{eqn:centralll}
1\ra Z\ra G^c\stackrel{\pi_a}{\longrightarrow} G^a\ra 1
\end{equation}
and fix a primitive $g\in G^a$. 
Then there exists an isoclinic extension
$$
1\ra \tilde{Z}\ra \tilde{G}^c\stackrel{\tilde{\pi}_a}{\longrightarrow} G^a\ra 1
$$ 
and a lift $\tilde{g}^c\in \tilde{G}^c$ of $g$ 
such that the order of 
$\tilde{g}^c$ in $\tilde{G}^c$ equals the order of $g\in G^a$. 
\end{lemm}

\begin{proof}
Since $\rH^2(\Z/\ell^n, (\Q/\Z)^r)=0$ we can modify 
the sequence~\eqref{eqn:centralll} by an element in 
$\delta(\rH^1(G^a, (\Q/\Z)^r))$
to split the induced extension 
$$
1\ra Z\ra  Z_{g} \ra \langle g\rangle \ra 1, 
\quad \quad Z_{g}:= (\pi_a)^{-1}(\langle g\rangle ),
$$ 
this provides a lift of $g$ of the same order. 
\end{proof}

\begin{coro}
\label{coro:bock-trivial}
If for some $g\in Z$ the restriction of  $\pi_a $ to $\langle g\rangle $ defines an
embedding $\pi_a:\langle g\rangle \hookrightarrow G^a$ then
$G^c$ is isoclinic to $G^c/\langle g\rangle \times \langle g\rangle$.
\end{coro}

\begin{proof}
Consider the central extension
\begin{equation}
\label{eqn:zero}
1\ra \langle g\rangle\to G^c\to  G^c/\langle g\rangle\ra 1.
\end{equation}
It is defined by an element
$\alpha_{g}\subset \rH^2(G^c/\langle g\rangle ,\langle g\rangle)$. 
Consider the embedding $\langle g\rangle \subset \Q/\Z$. Then the 
image of $\alpha_{g}$ in 
$\rH^2(G^c/\langle g\rangle,\Q/\Z)$ is 0. 
Indeed, the extension is induced
from a central extensions of abelian quotient groups 
$$
1\ra \pi_a(\langle g\rangle)\to G^a\to G^a/\pi_a(\langle g\rangle)\ra 1,
$$
which is isoclinic to the trivial central extension;
the same holds for the exact sequence \eqref{eqn:zero}.
\end{proof}

\begin{lemm} 
\label{lemm:induced}
Consider a central extension of finite $\ell$-groups
$$
1\ra Z\ra G^c\ra G^a\ra 1,
$$ 
let $\alpha \in \rH^2(G^a,Z)$ be the corresponding class, 
and assume that there exists a $g\in G^a$  lifting
to a $g^c\in G^c$ of the same order.
Let $Z_{g}\subset G^c$ be the centralizer of $g^c$. 
Then the central extension 
$$
0\ra  \langle g^c\rangle \ra  Z_g\ra Z_{g}/\langle g^c\rangle \ra 0 
$$ 
is induced from the extension 
$$
0\ra \langle g \rangle \ra G^a \ra G^a/\langle g\rangle\ra 0. 
$$

If $V$ is a faithful representation of $G^c$ 
as in the proof of Lemma~\ref{lemm:explicit}
then $V/Z_g$ is a vector bundle over 
$\tilde{V}/(Z_g/\langle g^c\rangle)$, for some faithful representation $\tilde{V}$ of  
$Z_g/\langle g^c\rangle$. Furthermore, if $K:=k(V/Z_g)$ and 
$$
s^2_K:\rH^2(G^a,\Z/\ell^n)\ra \rH^2(G_K,\Z/\ell^n)
$$ 
is the corresponding homomorphism, then 
$s^2_K(\alpha)$ is induced from $G^a/\langle g\rangle$.
\end{lemm}

\begin{proof}
Immediate from the definitions. 
\end{proof}

\begin{lemm}
\label{lemm:action}
Consider a central extension of finite groups
$$
1\ra Z\ra G^c\stackrel{\pi_a}{\lra} G^a\ra 1
$$
and let $V=\oplus_j V_j$ be
a faithful representation of $G^c$ as in 
Lemma~\ref{lemm:explicit}, i.e., each $V_j=\Ind_Z^{G^c}(\chi_j)$, 
where $\{\chi_j\}_{j\in J}$ is a basis of $\Hom(Z,k^\times)$. Let 
$
\P:=\prod_{j\in J} \P(V_j).
$
Then:
\begin{enumerate}
\item  
$G^a$ acts faithfully on $\P$.
\item  
For any subgroup $\sigma\subset G^a$ the subset of $\sigma$-fixed points 
$\P^\sigma\subset \P$ is nonempty if and only if $\sigma\in \bar{\Sigma}(G^c)$.
\item  
Each irreducible component of $\P^\sigma$ is a product 
of projective subspaces of $\P(V_j)$, corresponding to different eigenspaces
of $\sigma$ in $V_j$, and distinct irreducible components are disjoint. 
\item Each irreducible component of $\P^\sigma$ is stable under the action of 
$H_{\sigma}\subset G^c$, 
the maximal subgroup  such that $[H_{\sigma},\pi_a^{-1}(\sigma)]=0$ in $G^c$;  
the action of $G^c/H_{\sigma}$ on the set of components of $\P^\sigma$ is free.
\item The action of $G^a$ on 
$\P^\circ:=\P\setminus \cup_{\sigma\in \bar{\Sigma}} \P^\sigma$ is free.
\end{enumerate}
\end{lemm}

\begin{proof}
Since the order of $G^c$ is coprime to the characteristic of $k$, 
every $g\subset G^c$ is semi-simple and we can decompose 
$$
V_j = \oplus_i V_j(\lambda_i(g)),
$$ 
as a sum of eigenspaces.
The subset of $g$-fixed
points splits as a product $\prod_{ij} \P(V_j(\lambda_i(g)))$, where
the product runs over different eigenvalues in different 
$V_j$. It follows that the subset of $g$-fixed points $\P^g\subset \mathbb P$ is a union of
products of projective subspaces of $\P(V_j)$.

If $\sigma\in\bar{\Sigma}$ then its elements
can be simultaneously diagonalized. Hence
the subset of fixed points in $\mathbb P=\prod_j \P(V_j)$
is a union of products of projective subspaces, and there is a Zariski open
subvariety of $\mathbb P$ stable under the action of the 
preimage $\tilde{\sigma}:=\pi_a^{-1}(\sigma)\subset G^c$. 

Let $\sigma:=\langle g,h\rangle\subset G^a$ be a subgroup such that $\sigma\notin \bar{\Sigma}$. 
Then the same holds for the images of $g,h$ 
in $\GL(V_j)$, for at least one $j\in J$. Thus the commutator
$[g,h]\in \GL(V_j)$ is a nontrivial scalar matrix, hence they
have no common eigenvectors, i.e.,  no common fixed points in $\P(V_j)$.
Thus if $\sigma\notin \bar{\Sigma}$ then $\sigma$ has no
fixed points in $\prod_j \P(V_j)$.
Note that projective subspaces corresponding to different eigenvalues
of $g$ do not intersect in $\P(V_j)$ and hence $\P^\sigma$
splits into a disjoint union of products of projective
subspaces of different $\P(V_j)$.

Assume that $[h,\tilde{\sigma}] =0$ in $G^c$ for some $h\in G^c$. Then 
$\langle h,\tilde{\sigma}\rangle$ has a fixed point in each component of $\P^\sigma$
and $h$ maps every component of $\P^\sigma$ into
itself. Thus a subgroup $H\subset G^c$, with $[H,\tilde{\sigma}] =0$ maps
every component of $\P^\sigma$ into itself.

Assume that $\langle h,\gamma\rangle \notin \bar{\Sigma}$, for some $\gamma\in \sigma$. Then
for some $j$, the images of $h,\gamma$ in $\GL(V_j)$ 
have nonintersecting invariant subvarieties in $\P(V_j)$.
In particular, $h$ does not preserve any component of $\P^\sigma$.
\end{proof}

\begin{lemm}
\label{lemm:XP}
Let $K=k(X)$ be a function field with Galois group $G_K$. Given 
a surjection $G_K\ra G^c$, onto some finite central extension of
an abelian group $G^a$, let
$\P=\prod_j \P(V_j)$ be the space constructed in Lemma~\ref{lemm:action}.
Then there is a rational map $\varrho:X\ra \P/G^a$ such that
\begin{itemize}
\item 
$\varrho$  maps the generic point of $X$ into ${\mathbb P}^\circ/G^a$;
\item 
the homomorphism 
$$
s_K:\rH^*_s(G^a,\Z/\ell^n)\ra \rH^*(G_K,\Z/\ell^n)
$$
factors through the cohomology of $\bP^\circ/G^a$.
 \end{itemize}
\end{lemm}

\begin{proof}
Let $X^\circ\subset X$ be an open affine subvariety such that 
$\pi_1(X^\circ)$ surjects onto $G^c$. Let $\tilde{X}^\circ \ra X^{\circ}$ be the induced 
unramified $G^c$covering. Then $k[\tilde{X}^\circ]$ decomposes into an infinite direct sum
of $G^c$-representations. Fix a point $\tilde{X}^\circ$ and consider its orbit.
The restriction of $k[\tilde{X}^\circ]$ to this orbit defines a regular
quotient $G^c$-representation isomorphic to $k[G]$; this admits homomorphisms to $k[V]$,
corresponding to maps $X\ra V/G^c$.

\end{proof}

\section{Basic valuation theory}
\label{sect:valuation}
 
Let $X$ be a variety over $k=\bar{\F}_p$, 
$K=k(X)$ its function field, and $G_K$ the absolute Galois group of $K$. 
We write $\Val_K$ for the set of
valuations of $K$ and 
$\DVal_K$ for the subset of divisorial valuations. 
The corresponding residue fields will be denoted by $\KK_{\nu}$. For $\nu\in \Val_K$,
let $D_{\nu}\subset G_K$ denote a decomposition group of $\nu$ and 
$I_{\nu}\subset D_{\nu}$ the inertia
subgroup; we have $G_{\KK_{\nu}}=D_{\nu}/I_{\nu}$.
The pro-$\ell$-quotients of these groups will be denoted by  $\G_K$,  $\D_{\nu}$, and 
$\I_{\nu}$, respectively. We will always assume that $p\neq \ell$. 
The corresponding abelianizations will be denoted by 
$\G^a_K, \D^a_{\nu}$, and $\I^a_{\nu}$;
their canonical central extensions by $\G^c_K, \D^c_K$, and $\I^c_K$. 
Under our assumptions, $\G^a_K$ is a free $\Z_{\ell}$-module of infinite rank.

\begin{lemm}
\label{lemm:how-used}
For $\nu\in \Val_K$ consider the commutative diagram

\

\centerline{
\xymatrix{ 
D_{\nu} \ar[r]\ar[d]_{\pi_{\nu}} &  G_K  \ar[d]^{\pi} \\
\D^a_{\nu} \ar[r]_{\delta^a_{\nu}} & \G^a_K,
}
}

\noindent 
where $\pi_v$ and $\pi$ are the canonical projections and $\delta^a_{\nu}$
is the induced homomorphism.  Then  $\delta^a_{\nu}$
is injective with primitive image. 
In particular, $\delta^a_{\nu}$ embeds $\I^a_{\nu}$ 
as a primitive subgroup of $\D^a_{\nu}$. 
\end{lemm}

Let $\Sigma(\G^c_K)$ be the set of primitive topologically 
noncyclic subgroups of $\G^a_K$
whose preimage in $\G^c_K$ is abelian. By \cite[Section 6]{bt-commute}, we have:

\begin{thm}
\label{thm:rank}
Assume that $\dim(X)\ge 2$. Then 
$$
\rk_{\Z_{\ell}}(\sigma)\le \dim(X), \quad \text{ for all } \quad \sigma\in \Sigma(\G^c_K).
$$
\end{thm}

The following key result gives a valuation-theoretic 
interpretation of liftable subgroups in $\G^a_K$;
it is crucial for the reconstruction of 
function fields in \cite{bt0} and \cite{bt1}.

\begin{thm} \cite[Corollary 6.4.4]{bt-commute}
\label{theo:main-valu}
Assume that $\dim(X)\ge 2$ and
let $\sigma\in\Sigma(\G_K^c)$. 
Then there exists a valuation $\nu\in \Val_K$
such that $\I^a_{\nu}$ is a subgroup of $\sigma$ of $\Z_{\ell}$-corank at most one
and $\sigma\subseteq \D^a_{\nu}$.  
\end{thm}

We recall some background from valuation theory. 
For $\nu\in \Val_K$ let $\Gamma_{\nu} = \nu(K^\times)$ be its value group. 
We have a fundamental inequality
\begin{equation}
\label{eqn:fund}
\trdeg_k(K) \ge \trdeg_k(\KK_{\nu}) + \dim_{\Q}( \Gamma_{\nu}\otimes \Q).
\end{equation}

A valuation $\nu$ is called {\em algebraic of rank}  $r$ if its value group
$\Gamma_{\nu}:=\nu(K^\times)$ is isomorphic to $\Z^r$ and admits a filtration
$$
\Z^r\simeq \Gamma_{\nu}=
\Gamma_{\nu_r}\supsetneq \ldots \supsetneq \Gamma_{\nu_1}\simeq\Z
$$
by free abelian subgroups $\Gamma_{\nu_j}$ corank $j$ 
which are value groups of embedded compatible valuations $\nu_j$. 
Such valuations arise from flags of irreducible subvarieties 
$\mathfrak c_j$ of codimension $j$
$$
\mathfrak c_r\subsetneq  \mathfrak c_{r-1} \subsetneq \ldots 
\subsetneq \mathfrak c_1\subsetneq X.
$$
A valuation is called an {\em Abhyankar} valuation if 
equality holds in \eqref{eqn:fund} 
(alternatively, $\nu$ is said to be without transcendence defect); 
every algebraic valuation is an Abhyankar valuation.

Let $\nu$ be an Abhyankar valuation of $K$. 
By \cite[Theorem 3.4(a)]{knaf-ann},
$$
\Gamma_{\nu}\simeq \Z^r, 
$$
for some $r$. By \cite[Theorem 1.1]{knaf-ann}, for any 
$\mathcal F:=\{ f_1,\ldots, f_m\} \subset \mathfrak o_{\nu}$
there exists a projective model $X$ of $K$ such that 
\begin{itemize}
\item the center $\mathfrak c=\mathfrak c_X(\nu)$ of $\nu$ 
is a generically smooth subvariety of 
$X$ of dimension $\dim_{\Q}( \Gamma_{\nu}\otimes \Q)$;
\item there exists a regular parameter system 
$(a_1,\ldots, a_r)$ of $\mathcal O_{X,\mathfrak c}$   
such that each $f_j$, $j=1, \ldots, m$, is an 
$\mathcal O_{X,\mathfrak c}$-monomial in $\{a_1,\ldots, a_{r}\} $.
\end{itemize}
We say that $\nu$ admits a smooth monomial uniformization on 
$X$ with respect to $\mathcal F$. In particular, we have:

\begin{prop}
\label{prop:knaf--ann}
Let $K$ be a function field of an algebraic variety over $k=\bar{\F}_p$.
Let $\nu$ be an algebraic valuation of rank  $r$. 
Then there exists a projective model $X$ of $K$ over $k$ 
such that 
the centers $\mathfrak c_j$ of $\nu_j$ on $X$
are irreducible subvarieties of codimension $j$
which are smooth at the general point of $\mathfrak c_r$. 
\end{prop}

The following statement has been proved in 
\cite[Theorem 1.2 and Corollary 1.3]{knaf-adv}.

\begin{prop}
\label{prop:uniformize}
Let $K$ be a function field of an algebraic variety over 
$k=\bar{\F}_p$,  $\nu\in \Val_K$, and 
$\mathcal F=\{ f_1,\ldots,f_m\} \subset \mathfrak o_{\nu}$ a finite set of functions. 
Then there exists a finite 
separable Galois extension $\tilde{K}/K$, 
an extension of $\nu$ to a valuation $\tilde{\nu}\in \Val_{\tilde{K}}$, 
and a projective model $\tilde{X}$ of $\tilde{K}$ 
such that 
\begin{itemize}
\item[(1)] the extension of residue fields $\tilde{\KK}_{\tilde{\nu}}/ \tilde{\KK}_{\tilde{\nu}}$ is purely inseparable;
\item[(2)] $\Gamma_{\tilde{\nu}}/\Gamma_{\nu}$ is a finite $p$-group;
\item[(3)] $\tilde{\nu}$ admits a smooth monomial uniformization on 
$\tilde{X}$ with respect to $\mathcal F$. 
\end{itemize}
\end{prop}

\begin{prop}
\label{prop:main}
Let $K$ be a function field of an algebraic variety over 
$k=\bar{\F}_p$ and $\I^a_{\nu}, \D^a_{\nu} \subset \G_K^a$ the
abelianized inertia, resp. decomposition subgroup
of some valuation $\nu\in \Val_K$. 
Let 
$$
\gamma_K \,:\, \G^a_K\ra G^a
$$
be a continuous surjective homomorphism onto a finite abelian $\ell$-group $G^a$.
Assume that $\gamma_K(\I^a_{\nu})$ is nontrivial.  
Then for any primitive cyclic subgroup 
$I\subset \gamma_K(\I^a_{\nu})$ 
there exists a divisorial valuation $\nu'\in \DVal_K$
such that 
$$
I\subseteq \gamma_K(\I^a_{\nu'}) \quad \text{ and }\quad 
\gamma_K(\D^a_{\nu}) \subseteq \gamma_K(\D^a_{\nu'}).
$$
\end{prop}

\begin{proof}
Assume that $\ell^n$ annihilates 
$G^a = \oplus_j \Z/\ell^{n_j}$, i.e., $n_j\le n$.  
By Kummer theory, the homomorphism 
$$
\gamma_K\in \Hom(\G^a_K/\ell^n, G^a)= \oplus_j \Hom(\G^a_K, \Z/\ell^{n_j}) = \oplus_j 
K^\times / (K^{\times})^{\ell^{n_j}}
$$
is defined by a finite set of elements $\mathcal F=\{ f_j\}$, with 
$f_j\in K^{\times} / (K^{\times})^{\ell^{n_j}}$. 
We lift these to a set of elements of $K^\times$, 
denoted by the same letter, and we may assume that these are in 
$\mathfrak o_{\nu}$.     
We apply Proposition~\ref{prop:uniformize} and pass 
to a finite separable Galois extension $\tilde{K}$ of $K$ 
over which $\tilde{\nu}$, the extension of $\nu$, admits a 
smooth uniformization on a projective model 
$\tilde{X}$ of $\tilde{K}$ with respect to $\mathcal F$.
Properties  (1) and (2) in Proposition~\ref{prop:uniformize} 
insure that the image in $G^a$  
of the corresponding inertia and decomposition group is unchanged, i.e., 
we have a diagram

\

\centerline{
\xymatrix{
\I^a_{\tilde{\nu}} \ar@{=}[d] & \!\!\!\!\subseteq & \D^a_{\tilde{\nu}}\ar@{=}[d]& \!\!\!\! \subset & \G^a_{\tilde{K}} \ar[d] &  \\
\I^a_{\nu}                    & \!\!\!\!\subseteq & \D^a_{\nu}                   & \!\!\!\!\subset &  \G^a_{K} \ar@{>>}[r] & G^a  
}
}

\

\noindent
Thus we may assume that $K=\tilde{K}$. 
Let $\mathfrak c = \mathfrak c(\nu)$ be the center of the (the lift of) $\nu$ on $X$. 
The regular parameter system $(a_1,\ldots, a_r)$ of $\mathcal O_{X,\mathfrak c}$ 
defines an algebraic valuation $\nu'=(\nu_1',\ldots, \nu'_r)$ 
of rank $r$ of $K$ and 
$$
\rho(\I^a_{\nu})= \rho(\I^a_{\nu'})\subseteq \rho(\D^a_{\nu}) = \rho(\D^a_{\nu'}).
$$
We have reduced the proof to the case of algebraic valuations, 
where it is straightforward: 
If $\rho (\I^a_{\nu_1'})$ is nontrivial in $G^a$ 
then we are done since
$$
\rho(\I^a_{\nu_1'})\subset \rho (\I^a_{\nu'})= 
\rho (\I^a_{\nu_1'}) \subset \rho (\D^a_{\nu'}).
$$
Otherwise, we can assume that $\rho(\I^a_{\nu_1'})$ and 
$\rho (\I^a_{\nu_{j}'}) = 0, j>i$ but $\rho(\I^a_{\nu_{i}'})\neq 0$.
We blow up the general point of the center of $\nu_{i}'$ and let $E$
be the exceptional divisor, it is generically smooth  by  Proposition~\ref{prop:knaf--ann}.
Then we have
$$
0\neq \rho(\I^a_E) 
\subset \rho(\I^a_{\nu'}/\I^a_{\nu_{i+1}'})\subset \rho(\D^a_{\nu'})/\rho(\I^a_{\nu_{i+1}'})
\subset \rho(\D^a_E) 
$$ 
and $\rho(\I^a_E)$ coincides with $\rho(\I^a_{\nu_i'}/\I^a_{\nu_{i+1}'})$,
which is nontrivial by assumption.
\end{proof}

\begin{rema}
The proof is essentially an application 
of the general result that algebraic valuations 
(or Abhyankar valuations) are dense in the patch-topology on $\Val_K$. 
\end{rema}

\section{Liftable subgroups and their configurations}
\label{sect:liftable}

Let $K=k(X)$ be the function field of an algebraic variety over
$k=\bar{\F}_p$. In this section, we compare
the structure of the fan $\Sigma(\G^c_K)$ with fans 
in its finite quotients. 
Consider the canonical central extension
\begin{equation}
\label{eqn:ell-central}
1\ra \mathcal Z_K\ra \G^c_K\ra \G^a_K\ra 1.
\end{equation}

\begin{lemm}
\label{lemm:surj}
We have
$$
\mathcal Z_K=[\G^c_K,\G^c_K].
$$
\end{lemm}

\begin{proof}
This holds for function fields of curves since the corresponding 
pro-$\ell$-quotients of their absolute Galois groups are free. 
In higher dimensions,  $\G^a_K$ embedds
into the product $ \prod_E \G^a_E$, where $E$ 
ranges over function fields of curves $E\subset K$.
Under the projection to $\G^c_K\ra \G^a_E$, the center of $\G^c_K$ maps
to zero, hence the claim.
\end{proof}

\begin{lemm}
\label{lemm:canonical}
Consider commutative diagrams of continuous homomorphisms

\centerline{ 
\xymatrix{
1 \ar[r]&  \mathcal Z_K\ar[r] \ar@{>>}[d]& \G^c_K \ar[r] \ar@{>>}[d]^{\gamma^c_K} & \G^a_K \ar@{>>}[d]^{\gamma_K} \ar[r] & 1\\ 
           1\ar[r] & Z \ar[r] &   G^c    \ar[r] &  G^a \ar[r] & 1,
}
}

\noindent
where $G^c$ is finite, with  fixed surjective $\gamma_K$. 
Assume that $Z$ is the maximal quotient of $\cZ_K$  
giving rise to such a diagram. Then $G^c$ is unique, modulo
isoclinism.
\end{lemm}

\begin{proof}
Assume that $G_1^c,G_2^c$ are two such extensions of 
$G^a$ with $Z_1,Z_2$, respectively, and put 
$G:=G_1^c\times_{G^a} G_2^c$. We have a natural surjection $G\ra G^a$ and an inclusion 
$Z_1\times Z_2\hookrightarrow G$. Moreover, $[G,G]\subseteq Z_1\times Z_2$.  
By Lemma~\ref{lemm:surj}, $\cZ_K$ is generated by commutators in $\G^c_K$, thus 
$Z_1\times Z_2$ is also generated by commutators in $G$. It follows that $G=G^c$. 
If both projections $[G^c, G^c]\to Z_1 , Z_2$ are isomorphisms
then $G_1^c$ and $G_2^c$ are isoclinic. Otherwise, we have a contradiction to
the maximality assumption. Since $G^a$ is finite, $G^c$ is also finite, as it is bounded
by $\wedge^2(G^a)^*$.
%{\bf check the proof!}
\end{proof}

We proceed to investigate the properties of {\em fans} under such 
factorizations. Let 
$$
\gamma_K:\G^a_K\ra G^a
$$ 
be a continuous surjective homomorphism onto a finite group. We choose a maximal 
finite central extension $G^c$  of $G^a$ as in Lemma~\ref{lemm:canonical}.

\begin{coro}
\label{coro:intermediate-factor}
Given continuous surjective homomorphisms
\begin{equation}
\label{eqn:factorr}
\G^a_K \stackrel{\tilde{\gamma}_K}{\lra} \tilde{G}^a\stackrel{\gamma}{\longrightarrow} G^a,
\end{equation}
with $\tilde{G}^a$ a finite group, there is a unique (modulo isoclinism of lower rows) 
diagram of central extensions 

\

\centerline{
\xymatrix{ 
1\ar[r] & \cZ_K \ar[r] \ar[d]& \G^c_K \ar[d]^{\tilde{\gamma}^c_K} \ar[r] & \G^a_K \ar[d]^{\tilde{\gamma}_K} \ar[r]          & 1 \\
1\ar[r] & \tilde{Z} \ar[r] \ar[d]& \tilde{G}^c \ar[d]^{\gamma^c} \ar[r] & \tilde{G}^a \ar[d]^{\gamma} \ar[r]          & 1 \\
1\ar[r] & Z\ar[r]& G^c       \ar[r]     & G^a \ar[r] & 1 
}
}

\noindent 
with surjective $\tilde{\gamma}^c_K$, $\gamma^c$ and maximal $\tilde{Z}$, $Z$. 
\end{coro}

\begin{proof}
Evident. 
\end{proof}

We will use the following observation:

\begin{lemm}
\label{lemm:group}
Let $\G^a$ be a profinite abelian group and 
$$
\G^a\stackrel{\gamma_j}{\lra} \tilde{G}^a_j, \quad j=1,\ldots, n,
$$ 
a collection of continuous surjective homomorphisms onto finite groups.
Then there exists a continuous surjection 
$$
\gamma:\G^a\ra \tilde{G}^a
$$ 
onto a finite group such that 
each $\gamma_j$ factors through $\gamma$:
$$
\gamma_j: \G^a\stackrel{\gamma}{\lra} \tilde{G}^a\ra \tilde{G}^a_j.
$$ 

\end{lemm}

\begin{proof}
We can choose $\tilde{G}^a$ to be the image of $\G^a$ in the direct product
$$
\tilde{G}_1\times \cdots \times \tilde{G}_n.
$$
\end{proof}

We are interested in factorizations \eqref{eqn:factorr}, 
with finite $\tilde{G}^a$, preserving 
liftable subgroups and their configurations. 
Throughout we will be working with the canonical, modulo isoclinism, 
diagram as in Corollary~\ref{coro:intermediate-factor}, i.e., a factorization 
as in Equation~\eqref{eqn:factorr}
will canonically determine $\Sigma(\tilde{G}^c)$ and the set of $\Delta$-pairs in $\tilde{G}^a$, 
by Lemma~\ref{lemm:iso-delta-pair}.
Let 
$$
\Sigma_E(G^c):=\{ \sigma\in \Sigma(G^c) \, |\, \sigma = \gamma_K(\sigma_K), \quad
\text{ for some } \quad \sigma_K\in \Sigma(\G^c_K)\} 
$$
be the subset of {\em extendable} subgroups.

\begin{lemm}
\label{lemm:compact}
Given a continuous surjective homomorphism
$$
\gamma_K:\G^a_K\ra G^a
$$
onto a finite abelian group there exists a factorization
$$
\G^a_K\stackrel{\tilde{\gamma}_K}{\lra} \tilde{G}^a\stackrel{\gamma}{\lra} G^a, \quad \gamma_K=\gamma\circ
\tilde{\gamma}_K,
$$
with finite $\tilde{G}^a$, such that for all $\sigma\in \Sigma(G^c)$ we have: 
if $\sigma$ is nonextendable then there is no 
$\tilde{\sigma}\in \Sigma(\tilde{G}^c)$ with $\gamma(\tilde{\sigma})=\sigma$.

Moreover, this holds for every finite $\bar{G}^a$ fitting into
$$
\G^a_K\ra \bar{G}^a\ra \tilde{G}^a.
$$
\end{lemm}

\begin{proof}
First we prove the statement for one  nonextendable $\sigma$. 
Write 
\begin{equation}
\label{eqn:iota}
\G^a_K=\projlim_{\iota\in I} G^a_{\iota}, \quad \quad \gamma_{\iota \iota'}:G^a_{\iota} \stackrel{}{\lra} G^a_{\iota'}, 
\quad \iota' \preceq \iota,
\end{equation}
where the limit is over finite continuous quotients of $\G^a_K$. 
Assume that for all $\iota$, 
there is some $\sigma_{\iota}\in \Sigma(G^c_{\iota})$ surjecting onto $\sigma$; 
this implies that there exist such $\sigma_{\iota'}$, for all $\iota'\preceq \iota$, with 
$\gamma_{\iota \iota'}(\sigma_{\iota})=\sigma_{\iota'}$. 

%{\bf is the following true?}

By compactness of $\G^a_K$, there exists 
a closed liftable $\sigma_K\subset \G^a_K$ 
surjecting onto $\sigma$. This contradicts our assumption that $\sigma$ is nonextendable. 
Thus there is a required factorization
$$
\G^a_K\ra \tilde{G}^a\stackrel{\gamma}{\lra} G^a.
$$

Let 
$$
\{ \sigma_1,\ldots, \sigma_n\} =\Sigma(G^c)\setminus \Sigma_E(G^c).
$$ 
For each $j$, let 
$$
\G^a_K\stackrel{}{\lra} \tilde{G}^a_j\stackrel{\gamma_j}{\lra} G^a
$$ 
be the factorization constructed above.
Now we apply Lemma~\ref{lemm:group}, combined with Corollary~\ref{coro:intermediate-factor},
and obtain factorizations of $\gamma_j$:
$$
\G^a_K\ra \tilde{G}^a \ra \tilde{G}_j\ra G^a, \quad \quad \gamma: \tilde{G}^a\ra G^a,
$$  
Assume that there is some $j$ for which there exists a 
$\tilde{\sigma}\in \Sigma(\tilde{G}^c)$ surjecting onto $\sigma_j$. 
Then image $\tilde{\sigma}$ in $\tilde{G}_j$ must be liftable, 
contradicting the construction in the first part.  
\end{proof}

Let
$$
I\subseteq D\subseteq G^a.
$$
be a $\Delta$-pair (see Definition~\ref{defn:delta}). Throughout, 
we assume that $G^a$ arises as a finite quotient of 
the Galois group of some function field $G_K$, 
in particular, the corresponding $G^c$ is determined as in Lemma~\ref{lemm:canonical}, up to isoclinism.   
We say that $(I,D)$ is {\em extendable} 
if there exists a valuation $\nu\in \Val_K$ and subgroups
$$
I^a\subseteq \I^a_{\nu}, \quad D^a\subseteq \D^a_{\nu}
$$
such that 
$$
\gamma_K(I^a)=I, \quad \gamma_K(D^a)=D.
$$
Recall that a $\Delta$-pair $(\tilde{I},\tilde{D})$ is said to surject 
onto $(I,D)$ if 
$$
\gamma(\tilde{I})=I, \quad \gamma(\tilde{D})=D.
$$

We will need the following strengthening of Lemma~\ref{lemm:compact} 

\begin{prop}
\label{prop:delta-pairs}
Given a continuous surjective homomorphism
$$
\G^a_K\ra G^a
$$
onto a finite abelian group there exists a factorization
$$
\G^a_K\ra \tilde{G}^a\ra G^a, 
$$
with finite $\tilde{G}^a$, such that for all $\Delta$-pairs $(I,D)$ in $G^a$ we have:
if $(I,D)$ nonextendable 
then there is no $\Delta$-pair $(\tilde{I}, \tilde{D})$ in $\tilde{G}^a$
surjecting onto $(I,D)$. 
\end{prop}

\begin{proof}
As in the proof of Lemma~\ref{lemm:compact}, it suffices to establish  
the statement for one nonextendable 
$\Delta$-pair; indeed, there are only finitely many $\Delta$-pairs in $G^a$ and 
the same application of Lemma~\ref{lemm:group} will then establish
it for all.

Assume that there is no finite quotient of $\G^a_K$ with the desired property.
We start with a factorization 
$$
\G_K^a\ra \tilde{G}^a\stackrel{\gamma}{\lra} G^a
$$ 
such that $\tilde{G}^a$ satisfies the conclusions of 
Lemma~\ref{lemm:compact}, i.e., no $\sigma \in \Sigma(G^c)\setminus \Sigma_E(G^c)$
is the image of a $\tilde{\sigma}\in \Sigma(\tilde{G}^c)$. 

Let $(I,D)$ be a nonextendable $\Delta$-pair. By our assumption, there exists a 
$\Delta$-pair $(\tilde{I}, \tilde{D})$
in $\tilde{G}^a$ surjecting onto $(I,D)$. 
Choose representatives $g_1,\ldots, g_n\in D$ for $D/I$ and their preimages 
$\tilde{g}_j:=\gamma^{-1}(g_j)\in \tilde{D}$. Note that for each $j$,  
$$
\sigma_j:=\langle g_j, I\rangle\in \Sigma(G^c), \quad  
\tilde{\sigma}_j:=\langle \tilde{g}_j, \tilde{I}\rangle\in \Sigma(G^c). 
$$
and that $\tilde{\sigma}_j$ surject onto $\sigma_j$. 
By Lemma~\ref{lemm:compact} and our choice of $\tilde{G}^a$, 
all $\sigma_j$ are extendable. Moreover, 
$$
\tilde{I}=\cap_{j=1}^n \tilde{\sigma}_j. 
$$
We replace and rename the original $\tilde{D}$ by  
$$
\tilde{D}:=\cup_{j=1}^n \tilde{\sigma}_j.
$$
Then $(\tilde{I}, \tilde{D})$ is a $\Delta$-pair in $\tilde{G}^a$ surjecting onto $(I,D)$. 

Now we consider a projective system of finite continuous quotients 
$$
\G^a_K \ra\tilde{G}^a_{\iota}\ra \tilde{G}^a\ra G^a, \quad \quad 
\gamma_{\iota \iota'}:G^a_{\iota} \stackrel{}{\lra} G^a_{\iota'}, 
\quad \iota' \preceq \iota.
$$
Assume that for each $\iota$ 
there exists a 
$\Delta$-pair $(\tilde{I}_{\iota},\tilde{D}_{\iota})$
in $\tilde{G}^a_{\iota}$ surjecting onto $(I,D)$. Iterating the construction above, 
we construct, for each $\iota$, a collection of liftable subgroups 
$$
\tilde{\sigma}_{\iota, 1}, \ldots,  \tilde{\sigma}_{\iota,n}
$$
and a $\Delta$-pair $(\tilde{I}_{\iota},\tilde{D}_{\iota})$ of the form
$$
\tilde{I}_{\iota}=\cap_{j=1}^n \, \tilde{\sigma}_{\iota,j}, \quad \tilde{D}_{\iota}:=
\cup_{j=1}^n \, \tilde{\sigma}_{\iota,j},
$$
such that
\begin{itemize}
\item $(\tilde{I}_{\iota},\tilde{D}_{\iota})$ surjects onto $(\tilde{I}_{\iota'},\tilde{D}_{\iota'})$, 
for each $\iota'\preceq \iota$, 
\end{itemize}
and in particular onto $(I,D)$.
By compactness of $\G^a_K$ (see Lemma~\ref{lemm:compact}), there exist closed subgroups 
$$
\sigma_{K,1}, \ldots , \sigma_{K,n} \in \Sigma(\G^c_K);
$$
the closed subgroups
$$
\I^a:=\cap\,  \sigma_{K,j}, \quad \D^a:=\cup\,  \sigma_{K,j}
$$ 
of $\G^a_K$ surject onto $I$, resp. $D$. 
On the other hand, 
by the theory of liftable pairs (see \cite[Section 6]{bt-commute} and \cite[Corollary 4.3]{bt1}), 
there exists a valuation $\nu\in \Val_K$ such that 
$$
\I^a\subseteq \I^a_{\nu}, \quad \D^a\subseteq \D^a_{\nu}.
$$
This contradicts our assumption that $(I,D)$ is not extendable. 
\end{proof}

\section{Galois cohomology of function fields}
\label{sect:galois}

In \cite{bt0}, \cite{bt1} 
we proved that if $k=\bar{\F}_p$, with $p\neq \ell$, and $X$ is an algebraic variety over $k$ 
of dimension $\ge 2$ then $K=k(X)$ is encoded, up to purely inseparable extensions, 
by $\G^c_K$, the second lower series quotient of $\G_K$. 
Related reconstruction results have been obtained in 
\cite{pop}, \cite{mochizuki},  \cite{pop2}.

The proof of the Bloch--Kato conjecture by Voevodsky, Rost, and Weibel, 
substantially advanced our understanding of the relations between 
fields and their Galois groups, in particular, their Galois cohomology. 
Indeed, consider the diagram

\[
\centerline{
\xymatrix@!{
      & G_K\ar[dl]_{\,\,\,\pi_c} \ar[dr]^{\,\,\,\pi} & \\
 \G^c_K \ar[rr]_{\pi_a}  &   &           \G^a_K. \\
}
}
\]

\

\noindent
The following theorem relates the Bloch--Kato 
conjecture to statements in Galois-cohomology, with coefficients in 
$\Z/\ell^n$ (see also \cite{efrat-1}, \cite{chebolu},
\cite{pos}).   

\begin{thm} \cite{B-3}, \cite[Theorem 11]{bt-msri}
\label{thm:bk}
Let $k=\bar{\F}_p$, $p\neq \ell$, and $K=k(X)$ be the function field of an algebraic variety 
of dimension $\ge 2$. The Bloch--Kato conjecture for $K$ is equivalent to:
\begin{enumerate}
\item 
The map 
$$
\pi^* \colon \rH^*(\G^a_K, \Z/\ell^n)\to \rH^*(\Gal_K, \Z/\ell^n)
$$ 
is surjective and
\item 
$\
\Ker(\pi_a^*) = \Ker(\pi^*)$.
\end{enumerate}
\end{thm}

This implies that the Galois cohomology of the pro-$\ell$- quotient $\G_K$ 
of the absolute Galois group $G_K$ encodes important birational information of $X$. 
For example, in the case above, $\G^c_K$, and hence $K$, modulo purely-inseparable
extensions, can be recovered from the cup-products
$$
\rH^1(\G_K,\Z/\ell^n) \cup \rH^1(\G_K,\Z/ \ell^n) \ra 
\rH^2(\G_K,\Z/\ell^n), \quad n\in \N.
$$
From now on, we will frequently omit the coefficient ring $\Z/\ell^n$ from notation.

The first part of the Bloch-Kato theorem says that every $\alpha_K\in \rH^i(G_K)$ is
induced from a cohomology class $\alpha^a\in \rH^i(G^a)$ of 
some finite {\em abelian quotient} $G_K\ra G^a$. 
An immediate application of this is the following proposition:

\begin{prop}
\label{prop:bk-apply}
Let $\alpha_K\in \rH^i(G_K)$ be defined on a model $X$ of $K$ and induced from a continuous 
surjective homomorphism $\chi:\hat{\pi}_1(X)\ra G$ onto a finite group. Let $\alpha=\alpha_X\in \rH^i_{et}(X)$ 
be the class representing $\alpha_K$ on $X$. Then 
there exists a finite cover $X=\cup_j X_j$ by Zariski open subvarieties such that, for each
$j$, the restriction $\alpha_j:=\alpha|_{X_j}$ is induced from a continuous surjective
homomorphism $\chi_j: \hat{\pi}_1(X_j)\ra G^a$ onto a finite abelian group and 
a class $\alpha^a\in \rH^i_s(G^a)=\wedge^i(\rH^1(G^a))$. 
\end{prop}

\begin{proof}
We first apply the Bloch-Kato theorem to $V^\circ/G$ and find a Zariski open subset
$U:=U^\circ_{\alpha}$ such that the restriction $\alpha_U$ is as claimed, i.e., 
induced from a class $\alpha^a\in \rH^i_s(G^a)=\wedge^i(\rH^1(G^a))$, for some 
homomorphism $\chi_a:G_K\ra G^a$ to a finite abelian group. 
 
By Lemma~\ref{lemm:vg}, for every  $x\in X$ there exists a map $f=f_x$ such that 
$f(x)\subset U$ and the restriction of $\alpha$ to $f^{-1}(U)\subset X$ equals
$f^*(\chi_a^*(\alpha^a))$. The claim follows by choosing a finite cover by open subvarieties
with these properties.   
\end{proof}

The second part implies the following:

\begin{coro}
\label{coro:indd}
Let $\alpha_K\in \rH^i(G_K)$. 
Assume that we are given finitely many quotients 
$$
\chi_j:G_K\ra G^a_j
$$ 
onto finite abelian groups 
and classes 
$$
\alpha_j^a\in \rH^i(G^a_j)
$$
with $\chi_j^*(\alpha_j^a)=\alpha_K$, for all $j$. 
Then there exists a continuous 
finite quotient $G_K\ra G^c$ onto a finite central extension 
of an abelian group $G^a$ such that 
\begin{itemize}
\item $\chi_j$ factor through $G^c$, i.e., there exist surjective homomorphisms
$\psi_j : G^c\ra G_j^a$ , for all $j$; 
\item there exists a class $\alpha^c\in \rH^i(G^c)$ with
$$
\alpha^c=\psi_j^*(\alpha_j^a), \quad \text{ for all } \quad j.
$$
\end{itemize}
\end{coro}

\begin{lemm}
\label{lemm:contain}
Let $Z\subset X$ be an affine subset, $Z=\Spec(\mathfrak o_Z)$, and $\nu\in \Val_K$. 
The general point of $\nu$ is contained in $Z$, 
$$
\mathfrak c_X(\nu)^{\circ}\subseteq Z
$$ 
if and only if
$\nu(f)\ge 0$, for all $f\in \mathfrak o_Z$. 
\end{lemm}

\begin{lemm}
\label{lemm:trivial-nu}
Let $X$ be a normal variety with function field $K$. Assume that $\alpha_K\in \rH^i(G_K)$ 
is defined on $X$ and induced from a homomorphism $\chi : \hat{\pi}_1(X)\ra G$ to a 
finite group $G$. Consider the sequence
$$
\chi_K: G_K\ra \hat{\pi}_1(X)\stackrel{\chi}{\lra} G.
$$
Then $\chi_K(I_{\nu})=0$, for every $\nu$ such that $\mathfrak c_X(\nu)^{\circ}\subset X$. 
\end{lemm}

\begin{proof}
An \'etale cover of $X$ induces an \'etale cover of the generic point of 
$\mathfrak c_X(\nu)$, thus the cover is unramified in $\nu$, i.e., $\chi_K(I_\nu)=0$.  
\end{proof}

\begin{coro}
Let $\alpha_K\in \rH^i(G_K)$. 
Let $X$ be a normal projective model of $K$ and $\cup_jX_j$ a 
finite cover by open subvarieties such that
$\alpha_K$ is defined on $X_j$,  for each $j$, 
and is induced from a class $\alpha_j^a\in \rH^i(G^a_j)$, via
a homomorphism $\chi_j:\hat{\pi}_1(X_j)\ra G^a_j$ to some finite abelian group. 
Then there exists a quotient 
$\chi_K:G_K\ra G^a$ and a class $\alpha^a\in \rH^i(G^a)$ such that
\begin{itemize}
\item
$\alpha^a$ induces $\alpha_K$ and
\item 
$\alpha^a$ is unramified on every extendable $\Delta$-pair $(I,D)$ in $G^a$.  
\end{itemize}
\end{coro}

\begin{proof}
Each $\alpha_j^a$ is unramified on all $\nu$ 
such that the generic point $\mathfrak c_X(\nu)^\circ\subset X_j$, 
by Lemma~\ref{lemm:trivial-nu}.
Since the classes $\alpha_j$ define the same element $\alpha_K\in \rH^i(G_K)$, there exists
an intermediate quotient $G^a$ and a class in 
$\alpha^a\in \rH^i(G^a)$ inducing $\alpha_K$. 
\end{proof}

\section{Unramified cohomology}
\label{sect:unram-cohomology}

An important class of birational invariants of algebraic varieties 
are \emph{unramified} cohomology groups, with finite constant coefficients 
(see \cite{bog-stable}, \cite{colliot-ojan}). These are defined as follows:
Let $\nu$ be a divisorial valuation of $K$. 
We have a natural homomorphism
$$
\partial_{\nu} \,:\,  \rH^i(G_K) \ra \rH^{i-1}(G_{\KK_{\nu}}).
$$ 
Classes in $\ker(\partial_{\nu})$ are called {\em unramified with respect to $\nu$}. 
The {\em unramified} cohomology is
$$
\rH^i_{nr} (G_K):=\bigcap_{\nu\in \DVal_K} 
\Ker(\partial_{\nu}) \subset \rH^i(G_K).
$$
For $i=2$ this is the {\em unramified Brauer group} which was used to 
provide counterexamples to Noether's problem, i.e.,  
nonrational varieties of type $V/G$,
where $V$ is a faithful representation of a finite group $G$ (see \cite{saltman}, 
\cite{bog-87}).

Generally, for $\nu\in \Val_K$ and $\alpha\in \rH^i(G_K)$ let
$$
\alpha_{\nu} \in \rH^i(D_{\nu})
$$
be the restriction of $\alpha$ to the decomposition subgroup $D_{\nu}\subset G_K$ of $\nu$.

\begin{lemm}
\label{lemm:unram-class}
A class $\alpha\in \Ker(\partial_{\nu})\subseteq \rH^i(G_K)$, 
for $\nu\in \DVal_K$, 
if and only if $\alpha_{\nu}$ is induced from the quotient $G_{\KK_{\nu}}=D_{\nu}/I_{\nu}$.
In particular, $\alpha_{\nu}$ is well-defined as an element in $\rH^i(G_{\KK_{\nu}})$. 
\end{lemm}

\begin{proof}
The exact sequence
$$
1\ra I_{\nu}\ra D_{\nu}\ra G_{\KK_{\nu}}\ra 1
$$
admits a noncanonical splitting, i.e., $D_{\nu}$ is noncanonical direct product of 
$G_{\KK_{\nu}}= D_{\nu}/I_{\nu}$ with the corresponding inertia group, which is   
a torsion-free central procyclic subgroup of $D_{\nu}$.
Thus 
$$
\rH^*(D_{\nu}) = \rH^*(G_{\KK_{\nu}})\otimes \wedge^* \rH^1(I_{\nu}).
$$
We have 
$$
\rH^1(I_{\nu}, \Z/\ell^n)= \rH^0(I_{\nu}, \Z/\ell^n)= \Z/\ell^n
$$
and
$$
\wedge^*(\rH^1(I_{\nu}, \Z/\ell^n))= \rH^1(I_{\nu}, \Z/\ell^n) \oplus \rH^0(I_{\nu}, \Z/\ell^n).
$$
Thus
$$
\rH^i(D_{\nu}) = \rH^{i-1}(G^a_{\KK_{\nu}})\otimes \wedge^*(\rH^1(I_{\nu}))  
\oplus  \rH^i(G^a_{\KK_{\nu}})
$$
and the differential $\partial_{\nu}$ coincides with the projection onto the first summand.
Hence $\partial_{\nu}(\alpha)= 0$ is equivalent to $\alpha_{\nu}$ being induced from 
$G^a_{\KK_{\nu}}=D_{\nu}/I_{\nu}$.
\end{proof}

\

Let $K$ be a function field over $k=\bar{\F}_p$. 
By Theorem~\ref{thm:bk}, a class $\alpha\in \rH^i(G_K)$ is induced from $\alpha_K^a\in \G^a_K$, which in turn is induced from
a finite quotient $\G_K^a\ra G^a$. 
Applying Proposition~\ref{prop:main} we obtain an alternative characterization unramified classes.
\begin{coro}
\label{coro:coro}
A class $\alpha_K\in \rH^i(G_K)$ is unramified if and only if for every $\nu\in \Val_K$ the restriction of 
$\alpha_K^c:=\pi_a^*(\alpha_K^a)\in \rH^i(\G_K^c)$ to $\D^c_{\nu}$ is induced from $\D^c_{\nu}/\I^c_{\nu}$.
\end{coro}

\begin{proof}
Assuming that for every $\nu\in \DVal_K$, the restriction of $\alpha_K^a$ to $\D^a_{\nu}$ 
is induced from $\D^a_{\nu}/\I^a_{\nu}$ we need to show this property for every $\nu$.  
This follows from Proposition~\ref{prop:main}.
\end{proof}

This allows to extend the notion of {\em unramified} to arbitrary valuations.

\

Combining the considerations above we obtain 
the notion of {\em unramified stable cohomology}
$$
\rH^*_{s,nr}(G)
$$
of a finite group $G$: a stable cohomology class $\alpha\in \rH^i_s(G)$ is 
unramified if and only if it is contained in the kernel of the
composition
$$
\rH^i_s(G) \ra \rH^i(G_K)\stackrel{\partial_{\nu}}{\lra} \rH^{i-1}(G_{\KK_{\nu}}),
$$
for every valuation $\nu\in \DVal_K$, where 
$K=k(V/G)$ for some faithful representation of $G$.
This does not depend on the choice of $V$, provided $\ell\neq {\rm char}(k)$. 
These groups are contravariant in $G$ and form a subring
$$
\rH^*_{s,nr}(G) \subset \rH^*_s(G).
$$ 
Furthermore:
\begin{itemize}
\item
If $V/G$ is stably rational then $\rH^i_{s,nr}(G) = 0$, for all $i\ge 2$. 
\item 
We have
$$
\rH^i_{s,nr}(G) \subseteq 
\rH^i_{s,nr}(\Syl_{\ell}(G))^{N_G(\Syl_{\ell}(G))}.
$$ 
\end{itemize}

\begin{rema}
In \cite{bpt}, we
proved that $\rH^i_{s,nr}(G)=0$, $i\ge 1$, 
for most quasi-simple groups of Lie type.
A complete result for quasi-simple groups and $i=2$ 
was obtained in \cite{kun}.

Note that 
the $\ell$-Sylow-subgroups of finite simple groups often 
have stably-rational fields of
invariants; 
this provides an alternative approach to 
our vanishing theorem.
\end{rema}

\begin{lemm}
\label{lemm:gc}
Let 
$$
1\ra Z\ra G^c\stackrel{\pi}{\lra} G^a \ra 1
$$
be a central extension of finite $\ell$-groups 
and $\alpha^a\in \rH^2_s(G^a)$. Then $\alpha^c:=\pi_a^*(\alpha)\in \rH^2_{s,nr}(G^c)$ if and only if for all 
$\sigma\in \Sigma:=\Sigma(G^c)$
the restriction of $\alpha^a$ to $\sigma$ is trivial, i.e., $\alpha^a\in \rR^2(\Sigma)$.  
\end{lemm}

\begin{lemm}
\label{lemm:unram-extend}
%\label{lemm:ext-unram}
Let 
$$
\gamma_K:\G^a_K\ra G^a
$$ 
be a continuous surjective homomorphism and 
$\alpha_K^a=\gamma_K^*(\alpha^a)\in \rH^i(\G^a_K)$, for some
$\alpha^a \in \rH^i_s(G^a)$. Assume that $\alpha_K^a$ is unramified and 
let $(I,D)$ be an extendable $\Delta$-pair in $G^a$. 
Then $\alpha^a$ is unramified with respect to $(I,D)$. 

Conversely, if $\alpha^a$ 
is unramified with respect to every extendable $\Delta$-pair in $G^a$ then
$\alpha_K^a\in \rH^i_{nr}(\G_K^a)$.
\end{lemm}

\begin{proof}

For the converse, for $\nu\in \DVal_K$, let $D:=\gamma_K(\D^a_{\nu})$ and $I:=\gamma_K(\I^a_{\nu})$. Then either $D$ is cyclic or 
$(D,I)$ is an extendable $\Delta$-pair in $G^a$. We have a commutative diagram

\

\centerline{
\xymatrix{
\rH^i_s(G^a)   \ar[d]_{\gamma_K^*} \ar[r]^{\phi} & \rH^i_s(D)      \ar[d]  & \ar[l]_{\psi} \ar[d] \rH^i_s(D/I)\\
\rH^i(\G^a_K)         \ar[r] & \rH^i(\D^a_{\nu})         & \ar[l] \rH^i(\D^a_{\nu}/\I^a_{\nu})
}
}

\

\noindent
In either case, $\alpha^a_K$ is unramified with respect to $\nu$, by Lemma~\ref{lemm:unram-delta}. 
\end{proof}

\section{Main theorem}
\label{sect:proof}

\begin{thm}
\label{thm:mainn}
Let $K=k(X)$ be a function field over $k=\bar{\F}_p$ of $\trdeg_k(K)\ge 2$ and 
$\alpha_K\in \rH^i_{nr}(G_K)$, with $\ell\neq p$ and $i>1$.
Then there exist a continuous homomorphism $G_K\ra \tilde{G}^a$ onto a finite $\ell$-group, 
fitting into a diagram

\centerline{
\xymatrix{
        &         &  G_K \ar[d] &   & \\
1\ar[r] & \tilde{Z}\ar[r] &  \tilde{G}^c \ar[r] & \tilde{G}^a \ar[r] & 1
}
}

\noindent
and a class $\tilde{\alpha}^a\in \rH^i(\tilde{G}^a)$ such that 
\noindent
\begin{enumerate}
\item $\alpha_K$ is induced from $\tilde{\alpha}^a$, 
\item $\tilde{\alpha}^c:=\pi_a^*(\tilde{\alpha}^a)\in \rH^*_{s,nr}(\tilde{G}^c)$.
\end{enumerate}
Conversely, every $\alpha_K\in \rH^i(G_K)$ 
induced from $\wedge^*(\rH^1(\tilde{G}^a))$ 
and unramified on some $\tilde{G}^c$ as above is in $\rH^i_{nr}(G_K)$.
\end{thm}

\

In this section we begin the proof of Theorem~\ref{thm:mainn}, reducing it 
to geometric statements addressed in Sections~\ref{sect:smooth} and \ref{sect:singular}.

\

Fix an unramified class 
$$
\alpha_K\in \rH^i_{nr}(G_K)\subset \rH^i(G_K).
$$
By Theorem~\ref{thm:bk}, we have a surjection 
$$
\pi^*: \rH^i(\G^a_K)\ra \rH^i(G_K), 
$$
let $\alpha_K^a\in \rH^i(\G^a_K)$ be the unramified class such that $\pi^*(\alpha_K^a)=\alpha_K$. 
Let 
$$
\gamma_K: \G_K^a\ra G^a
$$
be a continuous quotient onto a finite abelian $\ell$-group such that 
$\alpha_K^a$ is induced from a class $\alpha^a\in \rH^i_s(G^a)=\wedge^i(\rH^1(G^a))$.   
We have a diagram of central extensions:

\

\centerline{
\xymatrix{ 
1\ar[r] & \cZ_K \ar[r] & \G^c_K \ar@{>>}[d] \ar[r] & \G^a_K \ar@{>>}[d]^{\gamma_K} \ar[r]  & 1 \\
1\ar[r] & Z \ar[r]     & G^c\ar[r]^{\pi_a}                       & G^a    \ar[r]                      & 1 
}
}

\

\noindent
where the lower row is uniquely defined, up to isoclinism, as in Lemma~\ref{lemm:canonical}. 
The group $G^a$ might be too small, i.e., it may happen that 
$$
\pi^*_a(\alpha^a)\notin \rH^i_{s,nr}(G^c).
$$
Our goal is to pass to an intermediate finite quotient $\G_K^a\ra \tilde{G}^a$
fitting into a commutative diagram below

\

\centerline{
\xymatrix{ 
1\ar[r] & \cZ_K\ar[r] & \G^c_K \ar[d] \ar[r]     & \G^a_K \ar[d]^{\tilde{\gamma}_K} \ar[r]  & 1 \\
1\ar[r] & \tilde{Z} \ar[r]  & \tilde{G}^c \ar[d] \ar[r] & \tilde{G}^a \ar[d]^{\gamma} \ar[r]      & 1 \\
1\ar[r] & Z \ar[r]        &  G^c       \ar[r]                    & G^a \ar[r]                                & 1 
}
}

\

\noindent 
where the vertical arrows are surjections onto finite $\ell$-groups, and such that 
$\alpha^a_K$ is induced from a class $\tilde{\alpha}^a\in \rH^*(\tilde{G}^a)$ with 
$$
\tilde{\pi}_a^*(\tilde{\alpha}^a)\in \rH^i_{s,nr}(\tilde{G}^c).
$$ 

\

There are two possibilities:
\begin{enumerate}
\item There exists a finite quotient $G_K\ra G^a$ 
such that $\alpha_K$ is induced from $\alpha^a\in \rH^i_s(G^a)$
which is unramified with respect to every 
extendable $\Delta$-pair $(I,D)$ in $G^a$. 
This case is treated in Lemma~\ref{lemm:stable}.
\item On {\em every} finite quotient, the class $\alpha^a$ inducing $\alpha_K$
is ramified on {\em some} extendable $\Delta$-pair $(I, D)$. 
This possibility is eliminated by Lemma~\ref{lemm:impossible}. 
\end{enumerate}

\

\begin{lemm}
\label{lemm:stable}
Assume that $\alpha^a\in \rH^i_s(G^a)$ is unramified with respect to every 
extendable $\Delta$-pair $(I,D)$ in $G^a$. Then there is a factorization 
\begin{equation}
\label{eqn:factor}
\G^a_K \stackrel{\tilde{\gamma}_K}{\lra} \tilde{G}^a\stackrel{\gamma}{\lra} G^a, \quad
\gamma_K=\gamma\circ \tilde{\gamma}_K, 
\end{equation}
with finite $\tilde{G}^a$, 
such that
$$
\pi_a^*(\tilde{\alpha}^a)\in \rH^i_{s,nr}(\tilde{G}^c), \quad \quad  
\tilde{\alpha}^a:= \gamma^*(\alpha^a).
$$
\end{lemm}

\begin{proof}
Let $\tilde{G}^a$ be the quotient constructed in 
Proposition~\ref{prop:delta-pairs}, i.e., if $(I,D)$ 
is not an extendable $\Delta$-pair in 
$G^a$ then {\em no} $\Delta$-pair $(\tilde{I}, \tilde{D})$ 
surjects onto $(I,D)$. 
Thus, for each $\Delta$-pair $(\tilde{I}, \tilde{D})$ in $\tilde{G}^a$ one of the following holds:
\begin{itemize}
\item 
either $\gamma(\tilde{I})=0$ or $\gamma(\tilde{D})$ cyclic,
\item 
or $(\gamma(\tilde{I}),\gamma(\tilde{D}))$ is an extendable $\Delta$-pair in $G^a$.  
\end{itemize}
By Lemma~\ref{lemm:unram-delta}, combined with Lemma~\ref{lemm:unram-extend},  
$\tilde{\alpha}^a:=\gamma^*(\alpha^a)$ is unramified with respect to $(I,D)$.
\end{proof}

\ 

\begin{lemm}
\label{lemm:impossible}
There exists a finite quotient $G_K\ra G^a$ such that $\alpha^a\in \rH^i(G^a)$ induces
$\alpha_K$ and is unramified on every extendable $\Delta$-pair $(I,D)$ in $G^a$. 
\end{lemm}

The proof of this lemma is presented in Section~\ref{sect:smooth}, in  
the case when $K$ admits a smooth projective model; 
a reduction to the smooth case is postponed until Section~\ref{sect:singular}.

\section{The smooth case}
\label{sect:smooth}

Let $X$ be a smooth projective irreducible variety over an 
algebraically closed field $k$ 
with function field $K=k(X)$. By the Bloch-Ogus theorem, 
there is an isomorphism
$$
\rH^i_{nr}(G_K)= \rH^0_{Zar}(X, \mathcal H^i_{et}(X)),
$$
where $\mathcal H^i_{et}$ is an \'etale cohomology sheaf 
(see also Theorem 4.1.1 in \cite{ct-lectures}).
In particular, a class $\alpha_K \in \rH^i_{nr}(G_K)$ can be represented by a
finite collection of classes $\{ \alpha_n\}_{n\in N}$, with 
$\alpha_n$ defined on some Zariski open affine $X_n\subset X$, 
with $X=\cup_n X_n$, such that
the restrictions of $\alpha_n$
to some common open affine subvariety
$X^\circ\subset \cap_n X_n$ coincide. We will need the following strengthening:

\begin{lemm}
\label{lemm:bloch-ogus}
Let $X$ be a smooth variety with function field $K$ and 
$S=\{ x_1,\ldots, x_r\}\subset X$ a finite set of points. 
Given a class $\alpha_K\in \rH^i_{nr}(G_K)$ there exist 
\begin{itemize}
\item a Zariski open subset
$U_S\subset X$, containing $S$, and 
\item a class $\alpha_S\in \rH^i_{et}(U_S)$
\end{itemize} 
such that $\alpha_K$ and $\alpha_S$ 
coincide on some dense Zariski open subset $U_S^\circ\subset U_S$, 
i.e., for every representation
of $\alpha_K$ by $\{\alpha_n\}_{n\in N}$ 
as above there exists some dense Zariski open   
$U_S^\circ\subset U_S$ 
such that the restrictions of $\alpha_S$ and all $\alpha_n$ to $U_S^\circ$
coincide. 
\end{lemm}

\begin{proof} 
Fix a representation of $\alpha_K$ by $\{ \alpha_n\}_{n\in N}$ as above; in particular
$\alpha_{n}-\alpha_{n'}=0$ on $X_n\cap X_{n'}$, for all $n,n'\in N$.  
Every point in $S$ is contained in some $X_n$; 
refining the cover by Zariski open subsets,
we can achieve that each point in $S$ is contained in exactly one such subset;
i.e., after relabeling, we may assume that $x_n\subset X_n$, for $n=1,\ldots, r$. 
Put 
$$
X_S^\circ:=\cap_{n=1}^r X_n\quad \text{and} \quad
X_n^\circ:= X_n \setminus \left[
\left(
(X\setminus X_n)\cup (X \setminus X_{n'})
\right) 
\cap (X\setminus X_S^\circ) \right]. 
$$
If $x_n\in X_n$ then $x_n \in X_n^\circ$. Moreover, 
$$
X_S^\circ=X_n^\circ\cap X_{n'}^\circ, \quad \text{for all} \quad n,n'\in [1,\ldots, r].
$$ 
Define
$$
U_S:=\cup_{n=1}^r X_n^\circ. 
$$
The restrictions of $\alpha_n$ to $X_n^\circ$ 
can be glued to a unique class $\alpha_S\in \rH^i_{et}(U_S)$, 
via the Mayer-Vietoris
exact sequence. This is the required class. 
\end{proof}

Fix a representation of $\alpha_K\in \rH^i_{nr}(G_K)$ by 
$\{\alpha_n\}_{n\in N}$ as above. 
Each class $\alpha_n\in \rH^i_{et}(X_n)$ is represented by a finite 
collection $\{ X_{nm}\}$ of affine charts $X_{nm}$, with $\cup X_{nm}=X_n$ and finite
\'etale covers 
$$
\psi_{nm} :\tilde{X}_{nm}\ra X_{nm},
$$
such that the restrictions $\alpha_{nm}:=\alpha_n|_{X_{nm}}$ are induced from 
homomorphisms $\chi_{nm}:\hat{\pi}_1(X_{nm})\ra G_{nm}$ onto finite groups. 
Proposition~\ref{prop:bk-apply} 
implies that there is further refinement of the cover by affine subcovers
$$
X=\cup_j X_j,
$$
such that for each $j$ there exist 
\begin{itemize}
\item a finite {\em abelian} group $G_j^a$
\item a surjection $\chi_{K,j}:G_K\ra G^a_j$, and
\item a class $\alpha_j\in \rH^i(G^a_j)$ inducing $\alpha$ via $\chi_{K,j}$. 
\end{itemize}  
Corollary~\ref{coro:indd} implies that there exists a finite quotient $\pi_c: G_K\ra G^c$ 
onto a central extension of an abelian group $G^a$
such that the projections $\chi_{K,j}$ factor through $G^c$ 
and the images of $\alpha_j$ in 
$\rH^i(G^c)$ coincide. In particular, $\alpha_K$ is induced from $\alpha^c\in \rH^i(G^c)$. 

We claim that $\alpha^c$ is unramified on every pair $(\pi_c(I_{\nu}), \pi_c(D_{\nu}))$, 
for $\nu\in \Val_K$. Indeed, for each $\nu$, $\chi_{K,j}(I_{\nu})=0$ on at least 
{\em one} of the charts $X_j$, thus $\alpha_j$ is induced from 
$\chi_{K,j}(D_{\nu})/\chi_{K,j}(I_{\nu})$; since $\alpha^c\in\rH^i(G^c)$ 
is induced from $\alpha_j$, we have the same property for $\alpha^c$, with respect to the 
pair $(\pi_c(I_{\nu}), \pi_c(D_{\nu}))$.

\

The rest of the argument is similar to the proof of Lemma~\ref{lemm:stable}. Let 
$G_K\ra G^a$ be an intermediate quotient surjecting onto each $G^a_j$ and
$$
G_K\ra \tilde{G}^a\ra G^a
$$ 
the intermediate finite
quotient constructed in Proposition~\ref{prop:delta-pairs}. 
In particular, the projection of every $\Delta$-pair in $\tilde{G}^a$ to $G^a_j$ is 
either of the form $(\pi_{a,j}(I_{\nu}), \pi_{a,j}(D_{\nu}))$, 
for some $\nu\in \Val_K$, or {\em cyclic}. 
Let $\tilde{G}^c$ be a central extension as in 
Corollary~\ref{coro:intermediate-factor}, surjecting onto $G^c$. 
We have classes $\tilde{\alpha}_j^a\in \rH^j(\tilde{G}^a)$, 
induced by $\alpha_j^a$ constructed above, and mapping to the same 
class $\tilde{\alpha}^c\in \rH^j(\tilde{G}^c)$. 

We claim that $\tilde{\alpha}^c\in \rH^j_{s,nr}(\tilde{G}^c)$. Indeed, for 
every $\Delta$-pair $(\tilde{I},\tilde{D})$ in $\tilde{G}^a$ either 
$\tilde{D}$ projects to a cyclic group in $G^a$ or it is extendable, i.e., image of some
$(I_{\nu},D_{\nu})$. In the first case, 
all elements $\tilde{\alpha}_j$ are unramified on $(\tilde{I},\tilde{D})$. 
In the second case, at least one of the $\tilde{\alpha_j}$  is unramified on it.

\section{Reduction to the smooth case}
\label{sect:singular}

In absense of resolution of singularities in positive characteristic, we reduce to 
the smooth case via the de Jong-Gabber alterations theorem (see \cite{illusie}): 
The Galois group $G_K$ contains a subgroup $G_{\tilde{K}}$ of finite index, 
coprime to $\ell$, such that
\begin{itemize}
\item 
The function field $\tilde{K}$ corresponding to $G_{\tilde{K}}$ admits a smooth model, i.e., 
there exists a finite cover
$$
\rho: \tilde{X}\ra X, 
$$
of degree $|G_K/G_{\tilde{K}}|$ with $\tilde{X}$ smooth and 
$\tilde{K}=k(\tilde{X})$.
\end{itemize}

Let $\alpha_K\in \rH^n_{nr}(G_K)$ be an unramified class. Its restriction 
$\alpha_{\tilde{K}}$ to a class in $\rH^n(G_{\tilde{K}})$ is also unramified.  
By results in Section~\ref{sect:smooth}, there exists a surjection
\begin{equation}
\label{eqn:reff}
G_{\tilde{K}}\ra \tilde{G}^c
\end{equation}
onto a finite abelian $\ell$-group such that $\alpha_{\tilde{K}}$ is unduced from a class in 
$\tilde{\alpha}^c\in \rH^n_{s,nr}(\tilde{G}^c)$.

\begin{lemm}
\label{lemm:store}
There exists a diagram

\

\centerline{
\xymatrix{
G_{\tilde{K}} \ar[d] \ar[r] & \tilde{G} \ar[d] \ar@{>>}[r]^{\tilde{\pi}_c} & \tilde{G}^c\\
G_K               \ar[r] &  G                            & 
}
}

\

\noindent
where the vertical arrows are injections, 
with image of index coprime to $\ell$, $\tilde{G}$ and $G$ are finite groups, and  
$\alpha_K$ is induced from an element $\alpha_G\in \rH^n_{nr}(G)$. 
In particular, $\Syl_{\ell}(G)\simeq\Syl_\ell(\tilde{G})$. 
\end{lemm}

\begin{proof}
Fix a finite continuous quotient $G_K\ra G'$ such that $\alpha_K$ is induced from some
$\alpha_{G'}\in \rH^i(G')$. 
Note that for every intermediate quotient
$$
G_K\ra G\ra G'
$$
there exists an $\alpha_G\in \rH^i(G)$ inducing $\alpha_K$. 
It suffices to find a sufficiently large $G$ such that the sujection \eqref{eqn:reff} factors through a subgroup of $G$. This is a 
standard fact in Galois theory. 

Since $\tilde{\alpha}^c\in \rH^i_{s,nr}(G^c)$ its image 
$\tilde{\alpha} \in \rH^i(\tilde{G})$ is also unramified.
Since the index $(G:\tilde{G})$ is coprime to $\ell$, 
and since the unramified $\tilde{\alpha}$ is 
induced from an element $\alpha_G\in \rH^i(G)$, 
$\alpha_G$ is also unramified, as claimed.  
\end{proof}

At this stage, we cannot yet guarantee that $G$ 
is a central extension of an abelian group, nor that it is an $\ell$-group. 
However, we know that
$$
\mathrm{tr}(\mathrm{res}_{G_K/G_{\tilde{K}}}(\alpha_K))=\alpha_K \in \rH^i_{nr}(G_K), 
$$
modulo multiplication by an element in $(\Z/\ell^n)^\times$.

We need the following version of resolution of 
singularities in positive characteristic:

\begin{thm} \cite{resolve}
\label{thm:resolve}
Let $G$ be a finite $\ell$-group and 
$Y$ a smooth variety with a generically free action of $G$. 
Then there exists a $G$-variety $\tilde{Y}$ with
a proper $G$-equivariant birational map $\tilde{Y}\ra Y$ such that 
$\tilde{Y}/G$ is smooth.
\end{thm}

We are very grateful to D. Abramovich for providing 
the reference and indicating the main steps of the proof in \cite{resolve}:
\begin{itemize}
\item Theorem VIII.1.1 gives an equivariant modification 
$Y'$ of $Y$ with a regular log structure on $Y'$ 
such that the action is {\em very tame}, i.e., 
the stabilizers in $G$ of points in $Y'$ 
are abelian and act as subgroups of tori in 
toroidal charts of the log structure $Y'$.
\item By Theorem VI.3.2, the quotient $Y'/G$ 
of a log regular variety by a very tame action is log regular.
\item By Theorem VIII.3.4.9, which is a step in Theorem VIII.1.1, 
a log regular variety has a resolution of singularities.
\end{itemize}

\

We return to the proof of Theorem~\ref{thm:mainn}.
Start with a suitable faithful representation $V$ of $G$, 
and thus of $\Syl_{\ell}(G)=\Syl_{\ell}(\tilde{G})$, and construct a diagram

\

\centerline{
\xymatrix{
  \bar{V}\ar[r]^{\!\!\!\!\!\pi_Y\hskip 1cm}\ar[d]_{\pi_G} & Y=\tilde Y/\Syl_\ell(G)\\
 \bar{V}/G            &   
}
}

\

\noindent
where $Y=\tilde{Y}/\Syl_{\ell}(G)$ is the smooth projective variety from 
Theorem~\ref{thm:resolve}, and $\pi_Y$ is a $\Syl_{\ell}(G)$-equivariant
map from a $G$-equivariant projective closure of $V$.
Let $L=k(Y)$ be the function field of $Y$. 
Given a class $\alpha_L\in \rH^i_{nr}(G_L)$ we 
have a covering of $Y=\cup_n Y_n$ by affine Zariski open subsets 
and a finite set of classes $\{\alpha_n\}$ representing $\alpha_L$, as
considered in Section~\ref{sect:smooth}.

Pick a point $v\in \bar{V}$. The image $S:=\pi_Y(G\cdot b)$ of its $G$-orbit 
is a finite set of points. By Lemma~\ref{lemm:bloch-ogus}, there exist 
a dense Zariski open subset $U_S$ and a class $\alpha_S\in \rH^i_{et}(U_S)$
coinciding with $\alpha_L$ on some dense Zariski open subset $U_S^\circ\subset U_S$. 
Its preimage 
$$
\bar{U}_S^\circ:=\pi_Y^{-1}(U_S^\circ)\subset \bar{V}
$$ 
is a dense Zariski open subset 
containing $G\cdot v$. Put 
$$
\bar{U}_v:=\cap_{g\in G}\,\, g(\bar{U}_S^\circ)\subset \bar{V},
$$
it is a $G$-stable dense Zariski open subvariety containing $v$. 
Its image $\pi_G(\bar{U}_v)\subset \bar{V}/G$ 
is a Zariski open subset containing $\pi_G(G\cdot v)$. 
Note that $\pi_v : \bar{U}_v/\Syl_\ell(G)\to U_S$
is a birational morphism to an open subset and $\pi_v^*(\alpha_S)$
is well-defined in \'etale cohomology of $\bar{U}_v/\Syl_\ell(G)$.
It follows that the trace 
$$
\mathrm{tr}_{\pi_G}(\pi_v^*(\alpha_S))\in \rH_{et}^i(\bar{U}_v/G)
$$ 
is well-defined and coincides with $(G:\Syl_\ell(G)) \cdot \alpha_S$ at
the generic point of $\bar{V}/G$.

Thus we have a covering of $\bar{V}/G$ by Zariski open subsets of the form 
$\bar{U}_v/G$, $v\in \bar{V}$, with cohomology classes representing 
$\alpha_K$ on each chart. 
There exists a finite subcovering by $\bar{U}_v/G$ with extensions of $\alpha_S$
to each  $\bar{U}_v/G$. We can now apply Proposition~\ref{prop:bk-apply} 
to produce a finite subcover
such that on each chart, 
the class is induced from homomorphisms onto finite {\em abelian} 
groups, and proceed as in Section~\ref{sect:smooth}. 

\

\bibliographystyle{plain}
\bibliography{bloch}
\end{document}